\documentclass{amsart}

\usepackage{quiver}
\usepackage[margin=1in]{geometry}

\usepackage{amsrefs}
\usepackage{hyperref}
\usepackage{amsfonts}
\usepackage{amssymb}
\usepackage{amsmath}
\usepackage{amsthm}
\usepackage[inline]{enumitem}
\usepackage{tikz-cd}
\usepackage{color}

\usepackage[textsize=tiny]{todonotes}

\newtheorem{theorem}{Theorem}[section]
\newtheorem{corollary}[theorem]{Corollary}

\newtheorem{definition}[theorem]{Definition}

\newtheorem{lemma}[theorem]{Lemma}

\newtheorem{fact}[theorem]{Fact}
\newtheorem{remark}[theorem]{Remark}

\begin{document}

\title{Kaplansky Classes and Stability}

\author{Sean Cox}
\email{scox9@vcu.edu}
\address{
Department of Mathematics and Statistics \\
Virginia Commonwealth University \\
1015 Floyd Avenue \\
Richmond, Virginia 23284, USA 
}

\date{\today}

\thanks{Partially supported by NSF grant DMS-2154141.  Thanks to Marcos Mazari-Armida and Jan Trlifaj for explaining their results in \cite{MA_Tflifaj_Decon} during the BLAST conference at Baylor University (May 2026).  Thanks also to the BLAST conference organizers and the NSF support of the conference (DMS 2519783).  The Claude Opus and Fable models provided valuable feedback on multiple drafts (including noticing a crucial error in an earlier draft), fixed an issue about $\kappa$-presentable objects in the proof of Theorem \ref{thm_CharacterizeAccessibleCats}, helped the author understand the relevant literature on quasicoherent sheaves, and suggested the proof of Lemma \ref{lem_quasicoherence_downward} appearing here (replacing an earlier, more complicated argument).}

\subjclass[2020]{ 03C45, 03C48, 03E75, 20M30, 20M50, 18A32 }

\begin{abstract}
Soon after the proof of the Flat Cover Conjecture around the year 2000, two related concepts were introduced for classes in Grothendieck categories: \emph{Deconstructible classes} and the strictly weaker \emph{Kaplansky classes}.  All commonly-studied Kaplansky classes, such as the class $\mathcal{FM}$ of Flat Mittag-Leffler modules in $R$-Mod and the class $\mathcal{D}$ of Drinfeld vector bundles in Qcoh($X$), satisfy a stronger property we introduce here: they are \emph{Uniformly Stationary Kaplansky} classes.  While such classes generally lack the key feature (existence of precovers) that make deconstructible classes so central to modern relative homological algebra, they often suffice for model-theoretic stability.  This is true even in the absence of the Amalgamation Property, with various restricted classes of morphisms, and in some non-additive settings.  For example, $\mathcal{FM}$ with pure embeddings, and $\mathcal{D}$ with (either categorical or geometric) pure embeddings, are stable in all sufficiently closed cardinals.

\end{abstract}

\maketitle


\section{Introduction}

Enochs' \emph{Flat Cover Conjecture}---that over every ring, every module has a flat cover---was proved around the year 2000 (\cite{MR1832549}, \cite{MR1798574}).  The proof spawned two important concepts:  \emph{deconstructible classes} \cite{MR2384838} of objects in a Grothendieck category, and the weaker notion of a \emph{Kaplansky class} \cite{MR1926201}. The properties are equivalent for classes closed under directed colimits, but are generally distinct.  The class $\mathcal{FM}$ of Flat Mittag-Leffler modules is always a Kaplansky class~\cite{MR2988573}, but fails to be deconstructible over non-perfect rings such as $\mathbb{Z}$~\cite{MR2900444}, which implies that a proposal of Drinfeld~\cite{MR2181808} in infinite dimensional algebraic geometry could not work.  Furthermore, over non-perfect rings, $\mathcal{FM}$ even fails to have the key property possessed by deconstructible classes (existence of precovers) which makes deconstructible classes so useful in relative homological algebra.  

We show that all commonly-studied Kaplansky classes, such as $\mathcal{FM}$ and Drinfeld vector bundles, satisfy an apparently stronger property we call \textbf{Uniformly Stationary Kaplansky} (Section \ref{sec_StatKaplansky}). Although these classes can still behave badly from the homological or geometric perspective---$\mathcal{FM}$ and the Drinfeld vector bundles being the prime examples---we prove they have the nice model-theoretic property of \emph{stability}.  This remains true even when paired with restrictive classes of morphisms, and in situations where the Amalgamation Property fails (see below):

\begin{theorem}\label{thm_MainApplications}

Suppose $\mathbf{K}$ is an isomorphism-closed class of objects in $R$-Mod and: 
\begin{enumerate*}[label=(\roman*)]

	\item\label{item_KappaRegularLargerR} $\kappa$ is a regular uncountable cardinal with $\kappa > |R|$;
	\item\label{item_LambdaLessKappaYetAgain} $\lambda^{<\kappa} = \lambda$; 

	\item\label{item_ClosureUnderPureExt_MainApp} $\mathbf{K}$ is closed under pure extensions, and 
	
	\item\label{item_BothKappaLambdaUST} $\mathbf{K}$ is both $<\kappa$-Uniformly Stationary Kaplansky and $\lambda$-Uniformly Stationary Kaplansky. 
	
\end{enumerate*}
Then 
\begin{equation*}
(\mathbf{K}, \text{Pure}), \ (\mathbf{K}, \mathbf{K}\text{-Pure}), \text{ and } (\mathbf{K}, \text{Mono})  
\end{equation*}
are each $\lambda$-stable, where $\mathbf{K}$-$\text{Pure}$ denotes the class of pure embeddings with cokernel in $\mathbf{K}$.  Furthermore:
\begin{enumerate}
	\item The same conclusion holds for the category Qcoh($X$) of quasicoherent sheaves over a quasicompact and semi-separated scheme $X$, with ``purity" interpreted geometrically (stalkwise), and even for categorical purity under additional technical assumptions on $\mathbf{K}$.  See Theorem \ref{thm_Precise_Qcoh} on page \pageref{thm_Precise_Qcoh}.

	\item The same conclusion holds in the category $S$-Act of actions by the monoid $S$, if $\kappa > |S|$, ``extension" means ``Rees extension", and (in the Pure and $\mathbf{K}$-Pure cases) assuming $S$ is linearly ordered.    

\end{enumerate}
\end{theorem}

The $(\mathbf{K}, \text{Pure})$ part of Theorem \ref{thm_MainApplications} strengthens the main result of Mazari-Armida and Trlifaj~\cite{MA_Tflifaj_Decon}, who obtained the same conclusion (in $R$-Mod) under the stronger\footnote{by Lemma \ref{lem_DeconImpliesUnifStatKap} } assumption that $\mathbf{K}$ is $<\kappa$-deconstructible.  By \cite[Example 2.25]{MA_Tflifaj_Decon}, such categories can fail to have the Amalgamation Property.

\begin{corollary}\label{cor_SpecificExamples}
Each of the following is stable at all sufficiently closed cardinals:
\begin{enumerate}
	\item In $R$-Mod: $\mathcal{FM}$ together with any of Mono, Pure, or $\mathcal{FM}$-Pure;
	
	\item In Qcoh($X$) where $X$ is a quasicompact semi-separated scheme:  $\mathcal{D}:=$Drinfeld vector bundles together with any of Mono, Pure, or $\mathcal{D}$-Pure (with ``pure" interpreted in either the categorical or geometric sense)
	
\end{enumerate}
\end{corollary}

The results above follow from a very general sufficient condition for stability (Theorem \ref{thm_MainThm_Stability} on page  \pageref{thm_MainThm_Stability}), which involves generalizing the Uniformly Stationary Kaplansky property from classes \emph{of objects}, to entire categories where there may not even be a good notion of a quotient.  The generalization is a variant of the downward L\"owenheim-Skolem (LS) property that we call the \textbf{$\boldsymbol{<\kappa}$-Uniformly Stationary LS property} (Definition \ref{def_Stat_LS}).  This property yields a new characterization of accessible categories (Theorem \ref{thm_CharacterizeAccessibleCats} on page \pageref{thm_CharacterizeAccessibleCats}).

Notation and conventions for $R$-Mod, $S$-Act, and Qcoh($X$) agree with \cite{MR1753146}, \cite{MR1751666}, and \cite{MR2139915}, respectively.  A monoid $S$ is a structure satisfying the group axioms with the possible exception of existence of inverses; and a (left) $S$-act on a set $X$ is defined just as group actions.  A monoid $S$ is (left) linearly ordered if for every $s, t \in S$, one of them is a (left) multiple of the other.  

\textbf{Comment on cardinal bound notation:}  we work with concepts from several fields that, unfortunately,  have inconsistent conventions for
cardinal bound notation.  
\begin{enumerate*}
	\item \textbf{Set-theoretic closure:} we write $<\!\lambda$-\emph{closed} for
	closure under sequences of length strictly less than $\lambda$, to avoid the
	ambiguity of ``$\lambda$-closed" in the set-theoretic literature (where it sometimes means closure under sequences of length $\lambda$, and elsewhere it means sequences of length strictly less than $\lambda$).

	\item \textbf{Category-theoretic bounds:} terms such as $\lambda$-presentable,
	$\lambda$-generated, $\lambda$-directed, and $\lambda$-pure are used with their
	standard meaning as in Ad\'amek--Rosick\'y~\cite{MR1294136}, in which the bound is
	already, by definition, with respect to subsets/subobjects of size $<\lambda$
	(e.g.\ a poset is $\lambda$-directed iff every subset of size $<\lambda$ has an
	upper bound). We do not add a ``$<$" to these terms, because it is not the
	convention in \cite{MR1294136}.

	\item \textbf{The Kaplansky property:} the original 
	$\lambda$-Kaplansky property (see page \pageref{eq_OrdinaryStatKap_NoModels}) was defined in terms of subsets and submodules of size $\lambda$, not $<\lambda$. 
	
\end{enumerate*}
We follow each source's own convention rather than reconcile them, and use ``$\le \lambda$ or ``$<\lambda$" to clarify when necessary.

\section{Stationary sets and Stationary Kaplansky classes}\label{sec_StatKaplansky}

\subsection{Stationary sets and classes}

The universe of sets is denoted $(V,\in)$.  For any set $N$, let $\mathfrak{N}$ denote the structure $(N,\in)$; here we really mean $N$ with the predicate $\{ (a,b) \ : \ a \in b, \ a \in N, \text{ and } b \in N   \}$.  For a metamathematical natural number $n$ we write $\mathfrak{N} \prec_n (V,\in)$ to mean that $\mathfrak{N}$ is elementary in the universe for all $\Sigma_n$ (and $\Pi_n$) formulas in the language of set theory.  The Levy-Montague Reflection Theorem, together with definability of $\Sigma_n$ truth (for fixed $n$, see \cite[Chapter 0]{MR1994835}), ensures that for any fixed $n$, there are plenty of such $\mathfrak{N}$.  For a cardinal $\theta$, $H_\theta$ denotes the set of those sets whose hereditary closure has cardinality $<\theta$.  $H_\theta$ is a set, and for uncountable $\theta$, $(H_\theta,\in) \prec_1 (V,\in)$.

For a regular uncountable cardinal $\kappa$ and a set $X$, $[X]^{<\kappa}$ (or $\wp_\kappa(X)$) denotes the collection of $<\kappa$-sized subsets of $X$.  The following notions were introduced by Jech:  a set $C \subseteq [X]^{<\kappa}$ is \textbf{closed unbounded in $\boldsymbol{[X]^{<\kappa}}$} if $C$ is $\subseteq$-cofinal in $[X]^{<\kappa}$, and $C$ is closed under unions of $\subset$-increasing chains of length $<\kappa$.  A set $S \subset [X]^{<\kappa}$ is \textbf{stationary in $\boldsymbol{[X]^{<\kappa}}$} if $S \cap C$ is nonempty for every closed unbounded $C \subset [X]^{<\kappa}$.  For $\kappa = \omega_1$, Kuecker~\cite{MR0457191} proved an extremely useful characterization of stationarity in $[X]^{<\omega_1}$: a set $S \subseteq [X]^{<\omega_1}$ is stationary in $[X]^{<\omega_1}$ if and only if for every $F: X^{<\omega} \to X$, there is an $N \in S$ that is closed under $F$.  For $\kappa \ge \omega_2$ the characterization is a bit more cumbersome, due to the (consistent) possibility of Chang's Conjecture holding, and there are two possibly distinct notions of stationarity:  Jech's stationarity defined above, which is what we will always mean by ``stationary", and another notion of stationarity due to Shelah.  See Foreman~\cite{MattHandbook} for a good comparison.  

$V$ will denote the universe of sets and $[V]^{<\kappa}$ will denote the (proper) class of all $<\kappa$-sized sets.  Given a class $\Gamma \subseteq [V]^{<\kappa}$, let's say \textbf{$\boldsymbol{\Gamma}$ is stationary in $\boldsymbol{[V]^{<\kappa}}$} if for every $<\kappa$-sized set $P$ there is some $N \in \Gamma$ such that 
\begin{equation}\label{eq_DefStatGamma}
P \subset \mathfrak{N} \prec_1 (V,\in) \text{ and } N \cap \kappa \text{ is an ordinal}.
\end{equation}
This is equivalent to asserting that there is a proper class of $\theta$ such that $\Gamma \cap [H_\theta]^{<\kappa}$ is stationary in the Jech sense.  The requirement that $N \cap \kappa$ be an ordinal is superfluous when $\kappa = \omega_1$.  The choice of $\Sigma_1$ elementarity in \eqref{eq_DefStatGamma} is irrelevant; the definition given above is equivalent to the one using, for any fixed $n\ge 1$, $\prec_n$ instead of $\prec_1$.\footnote{This is due to the fact that, for \textbf{fixed} metamathematical $n$, $\Sigma_n$ truth is definable in $(V,\in)$.  Combined with the L\'evy-Montague Reflection Theorem, this yields a closed unbounded class of cardinals $\theta$ such that $(H_\theta,\in) \prec_n (V,\in)$. }

\begin{fact}\label{fact_Equiv_Stat}
For a regular uncountable $\kappa$, a set $X$, and a set $S \subseteq [X]^{<\kappa}$:  $S$ is stationary in $[X]^{<\kappa}$ if and only if 
\[
\text{Lift}^V(S):= \{  N \in [V]^{<\kappa} \ :  \ N \cap X \in S  \}
\]
is stationary in $[V]^{<\kappa}$.
\end{fact}
\begin{proof}
The fact is well-known but we could not find an exact reference with this formulation.  Suppose $S$ is stationary in $[X]^{<\kappa}$. Fix any set $P$ of size $<\kappa$ and any regular $\theta$ with $P,X \in H_\theta$.  Then $\text{Lift}^{H_\theta}(S):= \{ N \in [H_\theta]^{<\kappa} \ : \ N \cap X \in S \}$ is stationary in $[H_\theta]^{<\kappa}$; so in particular it meets the club of $N \in [H_\theta]^{<\kappa}$ such that $P \cup \{ X \} \subset N \prec (H_\theta,\in)$ and $N \cap \kappa \in \kappa$.  Since $(H_\theta,\in) \prec_1 (V,\in)$, this implies $N \prec_{1} (V,\in)$ and shows that $\text{Lift}^V(S)$ is stationary in $[V]^{<\kappa}$.  For the other direction, suppose $S$ is a nonstationary subset of $[X]^{<\kappa}$; then there is a club $D \subset [X]^{<\kappa}$ with $D \cap S = \emptyset$.  If $\text{Lift}^V(S)$ were stationary in $[V]^{<\kappa}$ there would be an $N \in [V]^{<\kappa}$ with $\{ D, S,X \} \subset N \prec_1 (V,\in)$, $N \cap \kappa \in \kappa$, and $N \cap X \in S$.  But since $D$ is club in $[X]^{<\kappa}$ and $D,X \in N \prec_1 (V,\in)$, it follows by elementarity of $N$ and closed unbounded property of $D$ that $N \cap X \in D$.  So $D \cap S$ is nonempty, a contradiction.
\end{proof}


Borrowing some terminology from Shelah's Stationary Logic, we sometimes write
\[
\text{stat}_{<\kappa} N \ \  \Phi(N,\dots)
\]
to mean that $\big\{ N \in [V]^{<\kappa} \ : \ \Phi(N,\dots) \big\}$ is stationary in $[V]^{<\kappa}$, and similarly for $\text{stat}_{\le \kappa}$ (which is the same as $\text{stat}_{<\kappa^+}$).

For an infinite regular cardinal $\kappa$ and a set $N$, \textbf{$\boldsymbol{N}$ is $\boldsymbol{<\kappa}$-closed} if every $<\kappa$-sized subset of $N$ is also an \emph{element} of $N$; i.e., $[N]^{<\kappa} \subset N$.  Any $N \prec_0 (V,\in)$ is automatically closed under finite sequences, so the notion of $<\kappa$-closure for partially elementary submodels of the universe only has real content when $\kappa \ge \omega_1$.  

\begin{fact}\label{fact_LessKappaClosedStat}
Suppose $\kappa$ is an infinite regular cardinal.  Then for any cardinal $\lambda \ge \kappa$ with $\lambda^{<\kappa} = \lambda$, the $<\kappa$-closed structures are a stationary subclass of $[V]^{\lambda}$.  If, additionally, $\lambda$ is regular, $\kappa < \lambda$, and 
\begin{equation}\label{eq_LessKappaClosureAll}
\mu^{<\kappa} < \lambda \text{ for all }\mu < \lambda,
\end{equation}
 then the $<\kappa$-closed structures are stationary in $[V]^{<\lambda}$. \end{fact}
The fact is well-known but we sketch a proof of the first part.  Fix any set $P$ of size $\lambda$ and a cardinal $\theta$ with $\lambda, P \in H_\theta$ and $\theta^{<\kappa} = \theta$.  Consider the set structure $\mathfrak{A}:=(H_\theta,\in)$ which, as mentioned above, is always $\Sigma_1$-elementary in $(V,\in)$.  Use the Downward L\"owenheim-Skolem Theorem, together with the cardinal arithmetic assumption, to recursively build a $\subseteq$-increasing and continuous chain $\langle N_i \ : \ i \le \kappa \rangle$ of $\lambda$-sized elementary substructures of $\mathfrak{A}$ with $[N_i]^{<\kappa} \subset N_{i+1}$ for each $i<\kappa$.  Regularity of $\kappa$ then ensures $N_\kappa$ is $<\kappa$-closed.  The proof of the second part is similar, except each $N_i$ is of size $<\lambda$, \eqref{eq_LessKappaClosureAll} ensures the size of $N_{i+1}$ is $<\lambda$, and $\kappa < \lambda$ and regularity of $\lambda$ ensure $N_\kappa \in [V]^{<\lambda}$.

\subsection{Uniformly Stationary Kaplansky Classes}

For a ring $R$ and an infinite cardinal $\mu \ge |R|$, an isomorphism-closed class $\mathbf{K}$ of $R$-modules is a \textbf{$\boldsymbol{\mu}$-Kaplansky class} (\cite{MR1926201}, \cite{MR2988573}) if for every $K \in \mathbf{K}$ and every subset $X$ of $K$ with $|X| \le \mu$, there is a submodule $A$ of $K$ of size $\le \mu$ such that $A \supseteq X$, and both $A$ and $K/A$ are in $\mathbf{K}$.  Put another way, letting
\[
\mathbf{K}(K,\mu):=\left\{ A \in [K]^{\le \mu}  \ : \  A \text{ and } K/A \text{ are both in } \mathbf{K} \right\},
\]
a class $\mathbf{K}$ is $\mu$-Kaplansky if 
\begin{equation}
\forall K \in \mathbf{K}: \ \mathbf{K}(K,\mu) \text{ is $\subseteq$-cofinal in } [K]^{\le \mu}.
\end{equation}

As we'll see, all commonly-studied $\mu$-Kaplansky classes actually satisfy the stronger requirement that 
\begin{equation}\label{eq_OrdinaryStatKap_NoModels}
\forall K \in \mathbf{K}: \ \mathbf{K}(K,\mu) \text{ is stationary in } [K]^{\le \mu}
\end{equation}
in the sense of Jech discussed above.  A class $\mathbf{K}$ with property \eqref{eq_OrdinaryStatKap_NoModels} will be called a \textbf{$\boldsymbol{\mu}$-Stationary Kaplansky Class}.

In what follows, we often confuse a $\Sigma_1$-elementary subset $N$ of the universe with its associated structure $\mathfrak{N}  = (N,\in)$, mainly to typographically distinguish it from the algebraic objects under consideration.  Using Fact \ref{fact_Equiv_Stat} and the ``stat" quantifier convention introduced earlier, the $\mu$-Stationary Kaplansky Property of $\mathbf{K}$ is equivalent to:  
\begin{equation}\label{eq_NonDiagStatKap}
(\forall K \in \mathbf{K}) \ (\text{stat}_{\le \mu} \mathfrak{N} ) \left(  \mathfrak{N} \cap K \text{ and } \frac{K}{\mathfrak{N} \cap K} \text{ are both in } \mathbf{K} \right).
\end{equation}

In fact, we'll see that the commonly-studied $\mu$-Kaplansky classes satisfy the apparently stronger property obtained by reversing the quantifier order of \eqref{eq_NonDiagStatKap} and replacing $\forall K \in \mathbf{K}$ by $\forall K \in \mathfrak{N} \cap \mathbf{K}$:
\begin{equation}\label{eq_DiagonalStatKap}
(\text{stat}_{\le \mu} \mathfrak{N}) \ (\forall K \in \mathfrak{N} \cap \mathbf{K} ) \left(   \ \mathfrak{N} \cap K \text{ and } \frac{K}{\mathfrak{N} \cap K} \text{ are both in } \mathbf{K} \right).
\end{equation}

\noindent Notice the uniform nature of  \eqref{eq_DiagonalStatKap}: the stationary class no longer depends on the particular $K \in \mathbf{K}$ as it did in \eqref{eq_NonDiagStatKap}, and every $\mathfrak{N}$ in the stationary class in \eqref{eq_DiagonalStatKap} has the desired property with respect to every member of $\mathbf{K}$ that is an element of $\mathfrak{N}$.  We will see that property \eqref{eq_DiagonalStatKap} is extremely useful when one has to deal with many objects in $\mathbf{K}$ simultaneously.   A class $\mathbf{K}$ with property \eqref{eq_DiagonalStatKap} will be called a \textbf{$\boldsymbol{\mu}$-Uniformly Stationary Kaplansky class}, and similarly for $<\mu$ in place of $\mu$.  The Uniformly Stationary Kaplansky property trivially implies the Stationary Kaplansky property,\footnote{Since a given object $K$ is a member of almost every $\mathfrak{N}$ in the stationary class $\Gamma$.} and we conjecture they are not equivalent.  If $\Gamma$ is a stationary subclass of $[V]^{<\kappa}$ we will say that \textbf{$\mathbf{K}$ has the Kaplansky Property on $\boldsymbol{\Gamma}$} if
\[
(\forall \mathfrak{N} \in \Gamma) \ (\forall K \in \mathfrak{N} \cap \mathbf{K})  \ \left( \text{both } \mathfrak{N} \cap K \text{ and } \frac{K}{\mathfrak{N} \cap K} \text{ are in } \mathbf{K} \right).
\] 
So $\mathbf{K}$ has the $<\kappa$-Uniformly Stationary Kaplansky property if it has the Kaplansky Property on some stationary $\Gamma \subset [V]^{<\kappa}$.

While the (non-uniform) Stationary Kaplansky Property can be expressed as in \eqref{eq_OrdinaryStatKap_NoModels} without reference to partially elementary submodels of the universe, it is difficult to see how one would even rephrase the Uniformly Stationary Kaplansky property in \eqref{eq_DiagonalStatKap} without them (or without closely related concepts, such as Stationary Logic over large set fragments of $V$).  

Kaplansky classes and their stationary versions make sense in any Grothendieck category \cite{MR2342555}, and in certain non-additive categories, such as the category $S$-Act of (left) actions of a monoid $S$ (or more generally, presheaf categories).  In that setting, if $A$ is a subact of $K$, $K/A$ refers to the Rees quotient.

A class $\mathcal{D}$ of $R$-modules is $<\kappa$\emph{-deconstructible} for a regular $\kappa > |R|$ if it is the ``filtration closure" of its $<\kappa$-sized members.\footnote{We use the definition as in Mazari-Armida and Trlifaj~\cite{MA_Tflifaj_Decon} and other recent literature on the subject.  Older papers only required $\mathcal{D}$ to be contained in the filtration closure of its $<\kappa$-sized members.}  We will not need the exact definition here; the key point is that deconstructible classes are Kaplansky classes, as shown in  \cite[Lemma 6.7]{MR2900444}, and this remains true if we strengthen \emph{Kaplansky} to \emph{Uniformly Stationary Kaplansky}:
\begin{lemma}[Deconstructibility implies Uniformly Stationary Kaplansky]\label{lem_DeconImpliesUnifStatKap}
If $|R| < \kappa$ and $\mathcal{D}$ is a $<\kappa$-deconstructible class of $R$-modules, then $\mathcal{D}$ is $<\lambda$-Uniformly Stationary Kaplansky for all regular cardinals $\lambda \ge \kappa$.
\end{lemma}
\begin{proof}
Fix a regular $\lambda \ge \kappa$; then trivially $\mathcal{D}$ is also $<\lambda$-deconstructible.  By \cite[Theorem 1.1]{Cox_MaxDecon}, for all large enough $\theta$, the set
\[
\left\{ \mathfrak{N} \in [H_\theta]^{<\lambda} \  : \ \forall D \in \mathfrak{N} \cap \mathcal{D} \ \ \mathfrak{N} \cap D \text{ and } \frac{D}{\mathfrak{N} \cap D} \text{ are both in } \mathcal{D}  \right\} 
\] 
contains a closed unbounded subset of $[H_\theta]^{<\lambda}$.  In particular, it is stationary in $[H_\theta]^{<\lambda}$.
\end{proof}

While $<\kappa$-deconstructibility trivially implies $<\lambda$-deconstructibility for all $\lambda \ge \kappa$, we don't know whether the Uniformly Stationary Kaplansky Property has a similar upward transfer.  This is the reason for the additional assumption about $\lambda$ in the statement of Theorem \ref{thm_MainApplications}.  All $<\kappa$-Uniformly Stationary Kaplansky Classes we are aware of are also $\lambda$-Uniformly Stationary Kaplansky whenever $\lambda^{<\kappa} = \lambda$.

The class $\mathcal{FM}$ is non-deconstructible over, for example, the ring $\mathbb{Z}$.  We'll see (Lemma \ref{lem_FM_UnifStatKap}) that $\mathcal{FM}$ is Uniformly Stationary Kaplansky.  Together with Lemma \ref{lem_DeconImpliesUnifStatKap}, this tells us the Uniform Stationary Kaplansky property is strictly weaker than deconstructibility.

\section{Categories of $\mathcal{L}$-structures and restrictions to elementary submodels}

If $\mathcal{C}$ is a locally $\lambda$-presentable category and $\mathfrak{N} \prec (V,\in)$ has access to a representative set $\text{Pres}_\lambda(\mathcal{C})$ of the $\lambda$-presentable objects of $\mathcal{C}$, then one can use a restricted Yoneda embedding together with reflection functors in the other direction to define ``restrictions'' like $f \restriction \mathfrak{N}$ whenever $f \in \mathfrak{N} \cap \text{Mor} \mathcal{C}$, even if $\mathfrak{N}$ is not $<\lambda$-closed.  This was used in \cite[Section 4]{CoxKamsmaRosicky_cellular} to completely characterize cellular and cofibrant generation in locally presentable categories, by how such restrictions behave.  In the current paper, we choose to work exclusively with subcategories of $\mathcal{L}$-structures for some (sorted, possibly infinitary) first-order signature $\mathcal{L}$, where $f \restriction \mathfrak{N}$ has a more straightforward interpretation, at least if $\mathfrak{N}$ is closed enough to deal with possibly infinite arities in the $\mathcal{L}$ symbols (which is automatic if all the arities are finite).

If $\mathcal{L} = \left( S, (c_i)_i, (F_j)_j, (R_k)_k \right)$ is a sorted, possibly infinitary first order signature, $\textbf{Str} \boldsymbol{\mathcal{L}}$ denotes the category whose objects are the $\mathcal{L}$-structures, and whose morphisms are $\mathcal{L}$-homomorphisms.  For a ring $R$, the signature $\mathcal{L}_R$ of $R$-modules consists of the (1-sorted) signature of abelian groups augmented with a unary function symbol $\cdot_r$ for each $r \in R$.  For a monoid $S$, the signature $\mathcal{L}_S$ of $S$-acts consists of a unary function symbol $f_s$ for each $s \in S$.  Then $R$-Mod and $S$-Act can be viewed as full subcategories of $\text{Str} \mathcal{L}_R$ and $\text{Str} \mathcal{L}_S$, respectively.  In fact, any accessible category is equivalent to one of the form $\text{Mod}(T)$, viewed as a full subcategory of $\text{Str} \mathcal{L}$ for some basic theory $T$ in a sorted (possibly infinitary) signature $\mathcal{L}$ \cite{MR1294136}.

For a signature $\mathcal{L}$ and a set $N$ we say \textbf{$\boldsymbol{\mathfrak{N}=(N,\in)}$ is $\boldsymbol{\mathcal{L}}$-appropriate} if 
\begin{equation}\label{eq_Def_L_Appropriate}
\mathcal{L} \cup \{ \mathcal{L} \} \subset \mathfrak{N} \prec_1 (V,\in) \text{ and } \mathfrak{N} \text{ is } <\text{ar}(\mathcal{L}) \text{-closed},
\end{equation}
where $\textbf{ar}\boldsymbol{(\mathcal{L})}$ is the least infinite regular cardinal strictly bounding all arities of symbols from $\mathcal{L}$.  For example, if $\mathcal{L}$ is finitary (i.e. $\text{ar}(\mathcal{L}) = \omega$) then \emph{every} $\mathfrak{N} \prec_1 (V,\in)$ with $\mathcal{L} \cup \{ \mathcal{L} \} \subset \mathfrak{N}$ is $\mathcal{L}$-appropriate, because $\Sigma_1$ (or even $\Sigma_0$) elementary substructures of the universe are always closed under finite subsets.  On the other hand, if $\text{ar}(\mathcal{L}) = \omega_1$, $\mathcal{L}$-appropriateness of $\mathfrak{N}$ means $\mathcal{L} \cup \{ \mathcal{L} \} \subset \mathfrak{N} \prec_{1} (V,\in)$ and $[\mathfrak{N}]^\omega \subset \mathfrak{N}$; such an $\mathfrak{N}$ will necessarily have cardinality at least $2^\omega$.

If $\mathfrak{N}$ is $\mathcal{L}$-appropriate then by assumption, every $\mathcal{L}$-symbol is an element of $\mathfrak{N}$, and $\mathfrak{N}$ is closed under the arity of that symbol.  It follows that if $A$ is an $\mathcal{L}$-structure and $A$ is an \emph{element} of $\mathfrak{N}$, then for every $\mathcal{L}$-function symbol $\tau$, its interpretation $\tau^A$ is an element of $\mathfrak{N}$.  Furthermore, $\mathfrak{N} \cap A$ is closed under application of $\tau^A$ because if $\overline{a} \in (\mathfrak{N} \cap A)^{\text{arity}(\tau)}$ then $\overline{a}$ is an element of $\mathfrak{N}$; then because both $\tau^A$ and $\overline{a}$ are elements of $\mathfrak{N}$, so is $\tau^A(\overline{a})$.  Hence, $\mathfrak{N} \cap A$ is an $\mathcal{L}$-substructure of $A$.  In fact, the assumption $[\mathfrak{N}]^{<\text{ar}(\mathcal{L})} \subset \mathfrak{N}$ ensures that $\mathfrak{N} \cap A$ is an $L_{\text{ar}(\mathcal{L}),\text{ar}(\mathcal{L})}$-elementary substructure of $A$.  And if $f:A \to B$ is an $\mathcal{L}$-homomorphism with $f \in \mathfrak{N}$---by which we always mean $f$ (as a set of ordered pairs), $A$, and $B$ are all elements of $\mathfrak{N}$---then $f \restriction \mathfrak{N}$ is an $\mathcal{L}$-homomorphism from $\mathfrak{N} \cap A$ into $\mathfrak{N} \cap B$, yielding the following commutative square in $\text{Str} \mathcal{L}$.  

\begin{equation}\label{eq_S_square}
\begin{tikzcd}[ampersand replacement=\&]
	\& A \&\&\& B \\
	{\mathcal{S}(\text{Str}\mathcal{L}, f,\mathfrak{N}):} \\
	\& {\mathfrak{N} \cap A} \&\&\& {\mathfrak{N} \cap B}
	\arrow["{f \in \mathfrak{N} \cap \text{Mor} \left( \text{Str}\mathcal{L}\right)}", from=1-2, to=1-5]
	\arrow[hook, from=3-2, to=1-2]
	\arrow["{f \restriction \mathfrak{N} \in \text{Mor}\left( \text{Str} \mathcal{L} \right)}"', from=3-2, to=3-5]
	\arrow[hook, from=3-5, to=1-5]
\end{tikzcd}
\end{equation}

\begin{definition}\label{def_TraceInclusion}
An \textbf{$\boldsymbol{\mathcal{L}}$-trace inclusion} is an inclusion of the form $\mathfrak{N} \cap A \subset A$, where $\mathfrak{N}$ is $\mathcal{L}$-appropriate, $A$ is an $\mathcal{L}$-structure, and $A \in \mathfrak{N}$.  In other words---by considering $\text{id}_A: A \to A$, which is an element of $\mathfrak{N}$ if $A \in \mathfrak{N}$---an $\mathcal{L}$-trace inclusion is any vertical map in a square of the form $\mathcal{S} \left( \text{Str} \mathcal{L},f, \mathfrak{N} \right)$.  We may simply refer to these as \emph{trace inclusions} when $\mathcal{L}$ is clear from the context.
\end{definition}

$\mathcal{L}$-trace inclusions are always \emph{at least} $L_{\text{ar}(\mathcal{L}),\text{ar}(\mathcal{L})}$-elementary, because by definition we require $\mathfrak{N}$ to be $<\text{ar}(\mathcal{L})$-closed.  In fact, for any regular $\lambda$, if $\mathfrak{N}$ is $<\lambda$-closed and $\mathcal{L}$-appropriate, then the trace inclusions are $L_{\lambda,\lambda}$-elementary.

\section{The Uniformly Stationary L\"owenheim-Skolem property}

In order to state our main theorems in a way flexible enough to handle various classes of morphisms, and to deal with settings where no good analogue of ``quotient" is available, we generalize the Uniformly Stationary Kaplansky property from classes of objects, to arbitrary subcategories of $\text{Str} \mathcal{L}$.  In this setting the property more closely resembles the various L\"owenheim-Skolem (LS) properties from model theory, so we borrow that terminology.

\begin{definition}[Uniformly Stationary L\"owenheim-Skolem Property]\label{def_Stat_LS}
Suppose $\mathcal{K}$ is a (not necessarily full) subcategory of $\text{Str} \mathcal{L}$.  
\begin{itemize}
	\item A set $\mathfrak{N} = (N,\in)$ is \textbf{$\boldsymbol{\mathcal{K}}$-appropriate (in $\text{Str} \mathcal{L}$)} if it is $\mathcal{L}$-appropriate, and whenever $f \in \mathfrak{N} \cap \text{Mor} \mathcal{K}$, the square $\mathcal{S}(\text{Str}\mathcal{L}, f,\mathfrak{N})$ from page \pageref{eq_S_square} is in $\mathcal{K}$; i.e., \emph{if} the top arrow of $\mathcal{S}(\text{Str}\mathcal{L}, f,\mathfrak{N})$ is in $\mathcal{K}$, \emph{then} all the other arrows and objects in that square are in $\mathcal{K}$.

	\item If $\Gamma \subset [V]^{<\kappa}$ is a class of $\mathcal{L}$-appropriate structures, \textbf{$\boldsymbol{\mathcal{K}}$ has the LS Property on $\boldsymbol{\Gamma}$} means that every member of $\Gamma$ is $\mathcal{K}$-appropriate.

	\item $\mathcal{K}$ has the \textbf{$\boldsymbol{<\kappa}$-Uniformly Stationary LS Property} if it has the LS Property on some stationary subclass of $[V]^{<\kappa}$; i.e., 
	\[
(\text{stat}_{<\kappa} \mathfrak{N} ) ( \forall f \in \mathfrak{N} \cap \text{Mor} \mathcal{K}) \left( \text{ all objects and arrows in the square } \mathcal{S}(\text{Str} \mathcal{L}, f,\mathfrak{N}) \text{ are in } \mathcal{K} \right)
\]

\end{itemize}
\end{definition}

The \textbf{$\boldsymbol{\kappa}$-Uniformly Stationary LS Property} will just mean the \textbf{$\boldsymbol{<\kappa^+}$-Uniformly Stationary LS Property}.

Theorem \ref{thm_CharacterizeAccessibleCats} below isn't needed for any of the stability results, but it may be of independent interest.  A category is \textbf{$\boldsymbol{\kappa}$-accessible} if it has only set-many $\kappa$-presentable\footnote{An object $C$ is $\kappa$-presentable if for every $\kappa$-directed system $(C_i)_{i \in I}$, every morphism from $C$ to the colimit of the $C_i$'s factors uniquely through one of the $C_i$'s; see \cite{MR1294136} for details.} objects (up to isomorphism), is closed under $\kappa$-directed colimits, and every object is the colimit of some $\kappa$-directed system of $\kappa$-presentable objects.  Ad\'amek and Rosick\'y~\cite[Corollary 2.36]{MR1294136} showed that a category is accessible if and only if it is equivalent to some full subcategory $\mathcal{C}$ of some $\text{Str} \mathcal{L}$ such that $\mathcal{C}$ is closed under $\kappa$-directed colimits and $\kappa$-pure subobjects (for sufficiently large $\kappa$) in $\text{Str} \mathcal{L}$.  Theorem \ref{thm_CharacterizeAccessibleCats} shows that, in their characterization, one can replace closure under $\kappa$-pure subobjects with the $<\kappa$-Uniformly Stationary LS Property. In isolation, this property is apparently weaker than closure under $\kappa$-pure subobjects, at least for sufficiently closed $\kappa$ (see Lemma \ref{lem_EasyPureTraces}).

\begin{theorem}[Characterization of accessible categories]\label{thm_CharacterizeAccessibleCats}

Let $\mathcal{K}$ be any category.  The following are equivalent:
\begin{enumerate}
	\item\label{item_K_accessible} $\mathcal{K}$ is accessible.
	
	\item\label{item_equiv_full_sub}  $\mathcal{K}$ is equivalent to a full subcategory $\mathcal{C}$ of $\text{Str} \mathcal{L}$ for some $\mathcal{L}$, such that for some regular cardinal $\kappa$: 
	\begin{itemize}
	   \item $\mathcal{C}$ is closed under $\kappa$-directed colimits in $\text{Str} \mathcal{L}$, and 
	   \item $\mathcal{C}$ has the $<\kappa$-Uniformly Stationary LS Property in $\text{Str} \mathcal{L}$.   
	\end{itemize}

\end{enumerate}
\end{theorem}

\begin{proof}
\ref{item_K_accessible} $\implies$ \ref{item_equiv_full_sub}:  By \cite[Theorem 5.35]{MR1294136}, there is a regular $\lambda$ and an $L_\lambda$ ``basic" theory $T$ (say with signature $\mathcal{L}$) such that $\mathcal{K}$ is equivalent to the full subcategory $\text{Mod}(T)$ of $\text{Str} \mathcal{L}$, and $\text{Mod}(T)$ is closed under $\lambda$-directed colimits in $\text{Str} \mathcal{L}$.  Let $\kappa$ be any regular cardinal such that $\kappa \ge \lambda$ and $\mu^{<\text{ar}(\mathcal{L})} < \kappa$ for all $\mu < \kappa$,\footnote{For example, $\kappa$ could be the successor of $2^{\lambda + |\mathcal{L}| + \text{ar}(\mathcal{L})}$.} and note that $\text{Mod}(T)$ is also closed under $\kappa$-directed colimits in $\text{Str} \mathcal{L}$.  By Fact \ref{fact_LessKappaClosedStat} the class of $<\text{ar}(\mathcal{L})$-closed structures is stationary in $[V]^{<\kappa}$, and almost all of them contain (as a subset) the $<\kappa$-sized set $\mathcal{L} \cup \{ \mathcal{L} \}$.  So the $\mathcal{L}$-appropriate structures are stationary in $[V]^{<\kappa}$, and it suffices to show that any such $\mathfrak{N}$ is also $\text{Mod}(T)$-appropriate.   Let $f:A \to B$ be a $\text{Str} \mathcal{L}$ morphism between models of $T$, with $f \in \mathfrak{N}$ (recall this always means $f$, together with its domain and codomain, are elements of $\mathfrak{N}$).  By the remarks above, $f \restriction \mathfrak{N}$ is a $\text{Str} \mathcal{L}$-homomorphism, and the vertical trace inclusions in $\mathcal{S}\left( \text{Str} \mathcal{L}, f,\mathfrak{N} \right)$ are $L_{\kappa,\kappa}$-elementarity; hence $\mathfrak{N} \cap A$ and $\mathfrak{N}  \cap B$ are also models of $T$.  So the square $\mathcal{S}(\text{Str} \mathcal{L}, f, \mathfrak{N})$ is in $\text{Mod}(T)$.

\ref{item_equiv_full_sub} $\implies$ 
\ref{item_K_accessible}:  suppose $\mathcal{C}$ is a full subcategory of $\text{Str} \mathcal{L}$ with the $<\kappa$-Uniformly Stationary LS Property, and closed under $\kappa$-directed colimits in $\text{Str} \mathcal{L}$.  We prove that $\mathcal{C}$ is $\kappa$-accessible.  It has $\kappa$-directed colimits by assumption, since $\text{Str} \mathcal{L}$ has (all) colimits and by assumption $\mathcal{C}$ is closed under $\kappa$-directed colimits in $\text{Str} \mathcal{L}$.  

It remains to show that $\mathcal{C}$ has only a set of $\kappa$-presentable objects, and that every object in $\mathcal{C}$ is a $\kappa$-directed colimit of them.  We will show that
\begin{enumerate}
	\item\label{item_EveryObjKappaDirCol} every object in $\mathcal{C}$ is a $\kappa$-directed colimit (in $\mathcal{C}$) of $<\kappa$-sized members of $\mathcal{C}$; and
	\item\label{item_KappaPresentChar} the $<\kappa$-sized members of $\mathcal{C}$ are exactly the $\kappa$-presentable members of $C$ (and hence just a set).  
\end{enumerate}

Let $\Gamma \subset [V]^{<\kappa}$ be a stationary class witnessing the $<\kappa$-Stationary LS Property of $\mathcal{C}$.  Fix any object $C$ in $\mathcal{C}$.  Then almost every $\mathfrak{N} \in \Gamma$ has $C \in \mathfrak{N}$; let $\Gamma_{C}$ denote this stationary subclass.  Consider the set
\[
\mathbb{D}:=\big\{ D \ : \ D = \mathfrak{N} \cap C \text{ for some }  \mathfrak{N} \in \Gamma_{C} \big\}
\]
and observe that, by $\mathcal{C}$-appropriateness of members of $\Gamma_C$ and fullness of $\mathcal{C}$ in $\text{Str} \mathcal{L}$, every member of $\mathbb{D}$ is a $\mathcal{C}$-subobject of $C$.  And every member of $\mathbb{D}$ is of size $<\kappa$.  $C$ is the colimit of $(\mathbb{D}, \subset)$ in $\text{Str} \mathcal{L}$ (hence $\mathcal{C}$) because for every $c \in C$ almost every $\mathfrak{N} \in \Gamma_C$ has $c \in \mathfrak{N}$.  We need to verify that $(\mathbb{D},\subset)$ is $\kappa$-directed.  Given any collection $\{ D_i \ : \ i < \mu \}$ of members of $\mathbb{D}$ with $\mu < \kappa$, by stationarity of $\Gamma$ there is an $\mathfrak{N} \in \Gamma$ containing all the $D_i$'s as elements and subsets, so $\mathfrak{N} \cap C \in \mathbb{D}$ is above all of them.  This completes the proof of \eqref{item_EveryObjKappaDirCol}.

To see \eqref{item_KappaPresentChar}:  if $C$ is a $\kappa$-presentable object in $\mathcal{C}$ then by \eqref{item_EveryObjKappaDirCol} it is a $\kappa$-directed colimit $C = \text{colim}_{i \in I} C_i$ of $<\kappa$-sized subobjects.  By $\kappa$-presentability, $\text{id}_C: C \to C = \text{colim}_{i \in I} C_i$ factors uniquely through one of the $C_i$'s, implying $|C| \le |C_i| < \kappa$.  For the other direction, suppose $C$ is a $<\kappa$-sized member of $\mathcal{C}$ and that $\pi: C \to \text{colim}_{i \in I} D_i$ is a $\mathcal{C}$ morphism into the colimit of a $\kappa$-directed system $(D_i)_i$ in $\mathcal{C}$.  By assumption this colimit is also the $\text{Str} \mathcal{L}$-colimit, and it easily follows that $\pi$ factors uniquely in $\text{Str} \mathcal{L}$ through one of the $D_i$'s.  Fullness of $\mathcal{C}$ in $\text{Str} \mathcal{C}$ implies this factorization is in $\mathcal{C}$.  Hence, $C$ is $\kappa$-presentable in $\mathcal{C}$. 
\end{proof}

Suppose $\mathcal{C}$ is a subcategory of $\text{Str} \mathcal{L}$, and that $\mathcal{C}$ has pushouts; we do \textbf{not} assume these pushouts are the same as those computed in $\text{Str} \mathcal{L}$.  Assume $f \in \mathfrak{N} \cap \text{Mor} \mathcal{C}$ where $\mathfrak{N}$ is $\mathcal{C}$-appropriate; hence the square $\mathcal{S}(\text{Str}\mathcal{L}, f,\mathfrak{N})$ is in $\mathcal{C}$, and we can take the $\mathcal{C}$-pushout of the lower left span to obtain the following commutative diagram  in $\mathcal{C}$, where the lower left part is a $\mathcal{C}$-pushout square, and $f/^{\mathcal{C}}\mathfrak{N}$ is the unique $\mathcal{C}$-morphism from the pushout into $B$ making the diagram commute:

\[\begin{tikzcd}[ampersand replacement=\&]
	\& A \&\&\&\& B \\
	\\
	{\mathcal{D}(\mathcal{C},f,\mathfrak{N}):} \&\&\& {P(\mathcal{C},f,\mathfrak{N})} \\
	\\
	\& {\mathfrak{N} \cap A} \&\&\&\& {\mathfrak{N} \cap B}
	\arrow["{f \in \mathfrak{N} \cap \text{Mor} \mathcal{C}}", from=1-2, to=1-6]
	\arrow["{\widetilde{f \restriction \mathfrak{N}}^{\ \mathcal{C}}}", curve={height=18pt}, from=1-2, to=3-4]
	\arrow["{f/^{\mathcal{C}} \mathfrak{N}}", from=3-4, to=1-6]
	\arrow["\ulcorner"{anchor=center, pos=0.125, rotate=-90}, draw=none, from=3-4, to=5-2]
	\arrow["\subseteq", from=5-2, to=1-2]
	\arrow["{f \restriction \mathfrak{N}}"', from=5-2, to=5-6]
	\arrow["\subseteq"', from=5-6, to=1-6]
	\arrow[curve={height=-18pt}, from=5-6, to=3-4]
\end{tikzcd}\]

\begin{fact}[\cite{Cox_MaxDecon} for $R$-Mod, \cite{Cox_FCC_MonoidActs} for $S$-Act]\label{fact_WhatD_looks_like}
Suppose $\mathcal{C}$ is $R$-Mod (resp. $S$-Act) and $\mathcal{L}$ is the language of $R$-modules (resp. $S$-Acts).  Then every $\mathcal{L}$-appropriate structure is $\mathcal{C}$-appropriate.  Suppose $f \in \text{Mor} \mathcal{C}$ is an \textbf{inclusion}, and $f \in \mathfrak{N}$ where $\mathfrak{N}$ is $\mathcal{L}$-appropriate.  Then regarding the diagram $\mathcal{D}(\mathcal{C},f,\mathfrak{N})$:
\begin{enumerate}
	\item The pushout $P(\mathcal{C},f,\mathfrak{N})$ can be taken to be $A + (\mathfrak{N} \cap B)$ (resp. $A \cup (\mathfrak{N} \cap B)$) and all arrows in the diagram can be taken to be inclusions. 
	\item $B/A$ is an element of $\mathfrak{N}$, and the cokernel (resp. Rees quotient) of $f \restriction \mathfrak{N}$, i.e. $\frac{\mathfrak{N} \cap B}{\mathfrak{N} \cap A}$, is isomorphic to $\mathfrak{N} \cap \frac{B}{A}$.
	\item The cokernel (resp. Rees quotient) of $f/^{\mathcal{C}} \mathfrak{N}$, i.e. $B/P$, is isomorphic to $\frac{B/A}{\mathfrak{N} \cap (B/A)}$. 
\end{enumerate}
\end{fact}

\section{Stability and the Uniformly Stationary LS Property}

We define Galois types for subcategories $\mathcal{K}$ of $\text{Str} \mathcal{L}$ as in \cite[Section 2]{MR3914179}, except we don't assume the morphisms are monomorphisms.  Consider the class of all ordered pairs $(\overline{b},f)$ where $\overline{b}$ is a wellordered sequence in the codomain of $f$.  Define $(\overline{b},f) \sim_{\text{atomic}} (\overline{c},g)$ if $\text{dom}(f) = \text{dom}(g)$, $\text{lh}(b) = \text{lh}(c)=:\alpha$ and there is a commutative square
\[\begin{tikzcd}[ampersand replacement=\&,cramped]
	B \& D \\
	A \& C
	\arrow["{g'}", dashed, from=1-1, to=1-2]
	\arrow["f", from=2-1, to=1-1]
	\arrow["g"', from=2-1, to=2-2]
	\arrow["{f'}"', dashed, from=2-2, to=1-2]
\end{tikzcd}\]
in $\mathcal{K}$ such that $f'(c_i) = g'(b_i)$ for all $i < \alpha$.  Let $\sim$ be the equivalence relation generated by $\sim_{\text{atomic}}$.  The \textbf{Galois type of $\overline{b}$ w.r.t $f$}, which we'll denote by $\text{gtp}\left( \overline{b} ; f  \right)$, is the $\sim$ equivalence class of $(\overline{b},f)$, which is typically a proper class.  If $f$ is an inclusion, say from $A$ to $B$, we may write the more typical $\text{gtp}\left( \overline{b}/A, B \right)$ instead of $\text{gtp} \left( \overline{b};f \right)$.  The collection of all Galois types with a fixed domain $A$ is denoted $\textbf{gS}^{<\infty}(A)$.  For a fixed ordinal $\alpha$, $\sim_{\alpha,\text{atomic}}$, $\sim_\alpha$, and $\text{gS}^\alpha(A)$ denote the obvious restrictions to sequences of length $\alpha$ (and similarly $\sim_{<\alpha,\text{atomic}}$, $\sim_{<\alpha}$, and $\text{gS}^{<\alpha}(A)$ for restrictions to sequences of length $<\alpha$).  We will say $\mathcal{K}$ is \textbf{stable in the cardinal $\boldsymbol{\lambda}$} if for every $\mathcal{K}$-object $A$ with $|A|=\lambda$, $\text{gS}^1(A)$ has cardinality at most $\lambda$.  $\mathcal{K}$ is \textbf{stable} if it is $\lambda$-stable for some $\lambda \ge \text{ar}(\mathcal{L})$.

Next is a variant of a standard lemma in the AEC literature.

\begin{lemma}\label{lem_LS_implies_BoundedReps}
Suppose $\kappa \ge \text{ar}(\mathcal{L})$ is a cardinal and $\mathcal{K}$ is a subcategory of $\text{Str} \mathcal{L}$ with the $\kappa$-Uniformly Stationary LS Property.  If $A$ is a $\mathcal{K}$-object of size at most $\kappa$, then every member of $\text{gS}_{\mathcal{K}}^\kappa(A)$ has a representative whose codomain is of size at most $\kappa$.  
\end{lemma}

\begin{proof}
Fix a $\mathcal{K}$-object $A$ of size $\kappa$, a $\mathcal{K}$-morphism $f: A \to B$, and a sequence $\overline{b} = \langle b_i \ : \ i < \kappa \rangle$ in $B$.  The $\kappa$-Uniformly Stationary LS assumption yields an $\mathcal{K}$-appropriate $\mathfrak{N} \in [V]^\kappa$ such that 
\[
\{ f, A,B \} \cup A  \cup \{ b_i \ : \ i < \kappa  \} \subset \mathfrak{N}.
\]

Then $A = \mathfrak{N} \cap A$ and since $\mathfrak{N}$ is $\mathcal{K}$-appropriate, the entire diagram
\[\begin{tikzcd}[ampersand replacement=\&]
	A \&\& B \\
	\begin{array}{c} A=\\\mathfrak{N} \cap A \end{array} \&\& {\mathfrak{N} \cap B}
	\arrow["f", from=1-1, to=1-3]
	\arrow[shift right, no head, from=2-1, to=1-1]
	\arrow[no head, from=2-1, to=1-1]
	\arrow["{f \restriction \mathfrak{N}}"', from=2-1, to=2-3]
	\arrow["\subseteq"', from=2-3, to=1-3]
\end{tikzcd}\]
is in $\mathcal{K}$.  Also, $\overline{b}$ is contained in $\mathfrak{N} \cap B = \text{cod} \left(f \restriction \mathfrak{N}\right)$, and the latter has cardinality at most $|\mathfrak{N}|=\kappa$.  So $\left(\overline{b}, f \right)$ is atomically related to $\left( \overline{b}, f \restriction \mathfrak{N} \right)$.
\end{proof}

Recall $\text{ar}(\mathcal{L})$ was defined to be the least regular, infinite cardinal strictly bounding all the arities of the function and relation symbols.  Let $\boldsymbol{|\mathcal{L}|}$ denote the number of $\mathcal{L}$-symbols.  For example, if $R$ is an infinite ring and $\mathcal{L}_R$ is the language of $R$-modules, $\text{ar}\left( \mathcal{L}_R \right) = \omega$ and $\left\vert \mathcal{L} \right\vert=|R|$.  Theorem \ref{thm_MainThm_Stability} is the main theorem of the paper.

\begin{theorem}\label{thm_MainThm_Stability}
Suppose $\mathcal{C}$ is a subcategory of $\text{Str} \mathcal{L}$, and $\mathcal{K}$ is a subcategory of $\mathcal{C}$.  Suppose $\kappa$ is a regular uncountable cardinal such that $\kappa > \left\vert \mathcal{L} \right\vert$ and $\mu^{<\text{ar}(\mathcal{L})} < \kappa$ for all $\mu < \kappa$.\footnote{If $\text{ar}(\mathcal{L}) = \omega$ (i.e. $\mathcal{L}$ is finitary), as in all our main applications, the cardinal arithmetic requirement is automatic.}  Assume also that $\lambda$ is a (not necessarily regular) cardinal such that:
\begin{enumerate}[label=(\roman*)]
		
	\item\label{item_cardinal_arith} $\lambda^{<\kappa} = \lambda$

	\item\label{item_K_lambda_stat_LS} $\mathcal{K}$ has the $\lambda$-Uniformly Stationary LS Property

	\item\label{item_factoring}   \textbf{$\boldsymbol{\mathcal{C}}$ has the $\boldsymbol{<\kappa}$-Uniformly Stationary LS Property with factoring through $\boldsymbol{\mathcal{K}}$}, by which we mean:
	\begin{itemize}
		\item $\mathcal{C}$ has pushouts; and
		\item $\mathcal{C}$ has the LS property on a stationary class $\Gamma \subset [V]^{<\kappa}$ of $\mathcal{C}$-appropriate $\mathfrak{N}$ with the following additional property:  whenever $f \in \mathfrak{N} \cap \text{Mor} \boldsymbol{\mathcal{K}}$, the maps $\widetilde{f \restriction \mathfrak{N}}^{\ \mathcal{C}}$ and $f/^{\mathcal{C}} \mathfrak{N}$ from the $\mathcal{C}$-diagram $\mathcal{D}\left( \mathcal{C},f,\mathfrak{N} \right)$ also lie in $\mathcal{K}$:

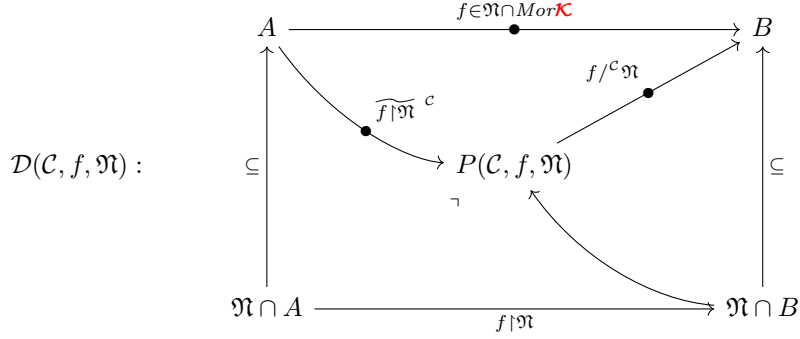
\[\begin{tikzcd}[ampersand replacement=\&]
	\& A \&\&\&\& B \\
	\\
	{\mathcal{D}(\mathcal{C},f,\mathfrak{N}):} \&\&\& {P(\mathcal{C},f,\mathfrak{N})} \\
	\\
	\& {\mathfrak{N} \cap A} \&\&\&\& {\mathfrak{N} \cap B}
	\arrow["{f \in \mathfrak{N} \cap \text{Mor} \color{red} \boldsymbol{\mathcal{K}}}"{inner sep=.8ex}, "\bullet"{marking}, from=1-2, to=1-6]
	\arrow["{\widetilde{f \restriction \mathfrak{N}}^{\ \mathcal{C}}}"{inner sep=.8ex}, "\bullet"{marking}, curve={height=18pt}, from=1-2, to=3-4]
	\arrow["{f/^{\mathcal{C}} \mathfrak{N}}"{inner sep=.8ex}, "\bullet"{marking}, from=3-4, to=1-6]
	\arrow["\ulcorner"{anchor=center, pos=0.125, rotate=-90}, draw=none, from=3-4, to=5-2]
	\arrow["\subseteq", from=5-2, to=1-2]
	\arrow["{f \restriction \mathfrak{N}}"', from=5-2, to=5-6]
	\arrow["\subseteq"', from=5-6, to=1-6]
	\arrow[curve={height=-18pt}, from=5-6, to=3-4]
\end{tikzcd}\]	

	\end{itemize}

	 \color{black}
	
	\item\label{item_ClosedUnderIso} The arrow category $\mathcal{K}^\to$ is closed under isomorphisms in the arrow category $\mathcal{C}^\to$.  I.e., if 
\[\begin{tikzcd}[ampersand replacement=\&]
	{} \& {} \\
	{} \& {}
	\arrow[from=1-1, to=1-2]
	\arrow["\simeq"', from=1-1, to=2-1]
	\arrow["\simeq", from=1-2, to=2-2]
	\arrow[from=2-1, to=2-2]
\end{tikzcd}\]
is a commutative square in $\mathcal{C}$ and the vertical arrows are $\mathcal{C}$-isomorphisms, then the top arrow is in $\mathcal{K}$ if and only if the bottom arrow is in $\mathcal{K}$. 

\end{enumerate}

Then $\mathcal{K}$ is $\lambda$-stable.  
\end{theorem}

\begin{proof}
Fix a $\mathcal{K}$-object $A$ of size at most $\lambda$.  Assumption \ref{item_K_lambda_stat_LS} and Lemma \ref{lem_LS_implies_BoundedReps} imply every member of $\text{gS}_\mathcal{K}^{\lambda}(A)$ has a representative whose codomain is of size at most $\lambda$; without loss of generality, whose codomain is a subset of $\lambda$.  We may also assume $A \subseteq \lambda$.  So to prove the theorem, we may restrict attention to the set
\[
X:= \left\{ \left( \overline{b}, f \right) \ : \ f \in \text{Mor}\mathcal{K}, \ \text{dom}(f) = A, \text{cod}(f) \subseteq \lambda, \ \text{ and } \overline{b} \in {}^{<\kappa} \text{cod}(f)   \right\} 
\]
and show that $X/ \sim^A_\mathcal{K}$ has cardinality at most $\lambda$.

There is a closed unbounded class of cardinals $\theta$ such that $H_\theta$ is closed under $\mathcal{C}$-pushouts, in the sense that for any $\mathcal{C}$-span that is an element of $H_\theta$, there exists a $\mathcal{C}$-pushout square for that span that is an element of $H_\theta$.\footnote{This can be shown either in ZFC if $\mathcal{C}$ is a definable class (using the Reflection Theorem and definability of $\Sigma_n$ truth in $(V,\in)$ for fixed $n$) or G\"odel-Bernays class theory.}  Fix such a $\theta$ with $X \in H_\theta$ and $\text{cf}(\theta) \ge \kappa$.  By assumption \ref{item_cardinal_arith} and Fact \ref{fact_LessKappaClosedStat}, there is a 
$\mathfrak{W} \prec (H_\theta,\in, \mathcal{C} \cap H_\theta,A,X,\lambda)$ such that $|\mathfrak{W}|=\lambda \subset \mathfrak{W}$, $A$ is both an element and subset of $\mathfrak{W}$, $\mathfrak{W}$ is $\mathcal{L}$-appropriate and 
\begin{equation}\label{eq_W_LessKappa}
[\mathfrak{W}]^{<\kappa} \subset \mathfrak{W}.
\end{equation}
Note that since $\mathfrak{W} \prec (H_\theta,\in,\mathcal{C} \cap H_\theta)$ and the latter is closed under $\mathcal{C}$-pushouts, then
\begin{equation}\label{eq_W_closed_C_pushouts}
\mathfrak{W} \text{ is closed under } \mathcal{C} \text{-pushouts.}
\end{equation}
We will show that every member of $X$ is $\sim_\mathcal{K}$ equivalent to a pair that is an \emph{element} of $\mathfrak{W}$.  This will complete the proof, since $|\mathfrak{W}| = \lambda$.

First we show:
\begin{equation}\label{eq_SmallSubstructuresInW}
B \in \text{Obj} \mathcal{K} \text{ and } \text{univ}(B) \subset \mathfrak{W} \implies \text{ every } < \kappa \text{-sized substructure of } B \text{ is an } \textbf{\emph{element of} } \mathfrak{W}. 
\end{equation}
Assume $B$ satisfies the hypothesis and $C$ is a substructure of $B$ of size $<\kappa$ (i.e., $|\text{univ}(C)|<\kappa$).  Then $\text{univ}(C) \in \mathfrak{W}$ by \eqref{eq_W_LessKappa}, but we also need to show that the remaining part of the structure is an element of $\mathfrak{W}$.  Consider any $\mathcal{L}$-function symbol $\tau$, with arity $\zeta$.  The function $\tau^C$ consists of ordered pairs of the form $\left( \overline{x}, y \right)$ where $\overline{x}$ is a $\zeta$-tuple in $\text{univ}(C)$, $y \in \text{univ}(C)$, and $\zeta < \kappa$ by assumption.  Because of \eqref{eq_W_LessKappa}, $\overline{x}$ is an \emph{element} of $\mathfrak{W}$.  And since $y$ is also an element of $\mathfrak{W}$, the pair $(\overline{x},y)$ is an element of $\mathfrak{W}$.  This shows $\tau^C \subset \mathfrak{W}$ (viewing $\tau^C$ as a set of ordered pairs).  Let $\mu:=|\text{univ}(C)|$.  The assumption that $\kappa > \mu^{<\text{ar}(\mathcal{L})} \ge \mu^\zeta$ ensures $|\tau^C|<\kappa$ and hence, by \eqref{eq_W_LessKappa} and the fact that $\tau^C \subset \mathfrak{W}$, $\tau^C$ is an \emph{element} of $\mathfrak{W}$.  Similar (but easier) arguments deal with constant and relation symbols, yielding that $\tau^C \in \mathfrak{W}$ for every $\mathcal{L}$-symbol $\tau$.  The assumption $\kappa > |\mathcal{L}|$, together with \eqref{eq_W_LessKappa}, then imply that the entire structure $C$ is an element of $\mathfrak{W}$, concluding the proof of \eqref{eq_SmallSubstructuresInW}.

Fix any $(\overline{b},f) \in X$; so $f: A \to B$ is $\mathcal{K}$-morphism and $A,B \subset \lambda$.  Let $\alpha:= \text{lh}(\overline{b})<\kappa$.  By assumption \ref{item_factoring} there is some $\mathfrak{N}$ such that $|\mathfrak{N}|<\kappa$, $\mathfrak{N}$ is $\mathcal{C}$-appropriate, $b_i \in \mathfrak{N}$ for every $i < \alpha$, $f \in \mathfrak{N}$, and the arrows labeled with a $\bullet$ in the following diagram lie in $\mathcal{K}$:

\begin{equation}\label{eq_MainBigDiag}
\begin{tikzcd}[ampersand replacement=\&]
	\begin{array}{c} A \\ ( = \mathfrak{W} \cap A) \end{array} \&\&\&\& \begin{array}{c} B \\ (= \mathfrak{W} \cap B) \end{array} \\
	\\
	\&\& {P=P\left(\mathcal{C},f,\mathfrak{N} \right)} \\
	\\
	{A_0:=\mathfrak{N} \cap A } \&\&\&\& {B_0:=\mathfrak{N} \cap B }
	\arrow["{f \in \mathfrak{N} \cap \text{Mor} \color{red} \boldsymbol{\mathcal{K}}}"{inner sep=.8ex}, "\bullet"{marking}, from=1-1, to=1-5]
	\arrow["{\widetilde{ f \restriction \mathfrak{N}}^{\ \mathcal{C}}}"'{inner sep=.8ex}, "\bullet"{marking}, curve={height=24pt}, from=1-1, to=3-3]
	\arrow["{f/^{\mathcal{C}}\mathfrak{N}}"'{inner sep=.8ex}, "\bullet"{marking}, from=3-3, to=1-5]
	\arrow["\ulcorner"{anchor=center, pos=0.125, rotate=-90}, draw=none, from=3-3, to=5-1]
	\arrow["\subset"{description}, from=5-1, to=1-1]
	\arrow["{f_0:=f \restriction \mathfrak{N}}"', from=5-1, to=5-5]
	\arrow["\subset"{description}, from=5-5, to=1-5]
	\arrow["j"', curve={height=-18pt}, from=5-5, to=3-3]
\end{tikzcd}
\end{equation}

Recall $b_i \in \mathfrak{N}$ for each $i < \alpha$, so $\overline{b}$ is a sequence in $B_0$.  Let $\overline{p}$ be the image of $\overline{b}$ via $j$.  Commutativity of \eqref{eq_MainBigDiag}, and the fact that $\widetilde{f \restriction \mathfrak{N}}^{\ \mathcal{C}}$ and $f/^\mathcal{C} \mathfrak{N}$ both lie in $\mathcal{K}$, ensure that 
\[
\left(\overline{p}, \widetilde{f \restriction \mathfrak{N}}^{\ \mathcal{C}} \right) \sim_{\mathcal{K},\text{atomic}}(\overline{b}, f).
\]

It remains to show that $\left( \overline{p},\widetilde{f \restriction \mathfrak{N}}^{\ \mathcal{C}} \right)$ is $\sim_{\mathcal{K}, \text{atomic}}$ related to an element of $\mathfrak{W}$.  By \eqref{eq_SmallSubstructuresInW}, both $A_0$ and $B_0$ are elements of $\mathfrak{W}$.  Then by \eqref{eq_W_LessKappa}, the map $f_0$ (viewed as a subset of $A_0 \times B_0$) is also an element of $\mathfrak{W}$.  Recall also that $A \in \mathfrak{W}$.  So the $\mathcal{C}$-span 

\[\begin{tikzcd}[ampersand replacement=\&,cramped]
	A \& \\
	{A_0} \& {B_0}
	\arrow["\subset", from=2-1, to=1-1]
	\arrow["{f_0}"', from=2-1, to=2-2]
\end{tikzcd}\]

\noindent ---which is just the lower left corner of the diagram \eqref{eq_MainBigDiag}---is an element of $\mathfrak{W}$.  Then by \eqref{eq_W_closed_C_pushouts}, $\mathfrak{W}$ has some $\mathcal{C}$-pushout square

\[\begin{tikzcd}[ampersand replacement=\&,cramped]
	A \& Q \\
	{A_0} \& {B_0}
	\arrow["h", from=1-1, to=1-2]
	\arrow["\ulcorner"{anchor=center, pos=0.125, rotate=-90}, draw=none, from=1-2, to=2-1]
	\arrow["\subset", from=2-1, to=1-1]
	\arrow["{f_0}"', from=2-1, to=2-2]
	\arrow["\ell"', from=2-2, to=1-2]
\end{tikzcd}\]
By the universality of pushouts there is a (unique) $\mathcal{C}$-isomorphism $i$ making the following diagram commute:

\begin{equation}\label{eq_PushoutIso}
\begin{tikzcd}[ampersand replacement=\&]
	\&\& P \\
	A \& Q \\
	{A_0} \& {B_0}
	\arrow["{\widetilde{f \restriction \mathfrak{N}}^{\ \mathcal{C}}}", curve={height=-12pt}, from=2-1, to=1-3]
	\arrow["h", from=2-1, to=2-2]
	\arrow["i", from=2-2, to=1-3]
	\arrow["\ulcorner"{anchor=center, pos=0.125, rotate=-90}, draw=none, from=2-2, to=3-1]
	\arrow["\subset", from=3-1, to=2-1]
	\arrow["{f_0}"', from=3-1, to=3-2]
	\arrow["j"', curve={height=12pt}, from=3-2, to=1-3]
	\arrow["\ell"', from=3-2, to=2-2]
\end{tikzcd}
\end{equation}

Assumption \ref{item_ClosedUnderIso},  together with the facts that $i$ is a $\mathcal{C}$-isomorphism and $\widetilde{f \restriction \mathfrak{N}}^{\ \mathcal{C}} \in \mathcal{K}$, ensure $h \in \mathcal{K}$.  And since $Q$ and $P$ are both $\mathcal{K}$-objects (being codomains of $\mathcal{K}$ morphisms) and $i$ is a $\mathcal{C}$-isomorphism, assumption \ref{item_ClosedUnderIso} ensures $i$ is a $\mathcal{K}$-morphism.\footnote{Apply assumption \ref{item_ClosedUnderIso} to the commutative square 
\[\begin{tikzcd}[ampersand replacement=\&,cramped]
	Q \& Q \\
	Q \& P
	\arrow["{\text{id}_Q}", from=1-1, to=1-2]
	\arrow["{\text{id}_Q}"', from=1-1, to=2-1]
	\arrow["i", from=1-2, to=2-2]
	\arrow["i"', from=2-1, to=2-2]
\end{tikzcd}\]} 
Let $\overline{q}$ be the image of $\overline{b}$ via $\ell$ ($=$ the preimage of $\overline{p}$ via $i$).  Since both $\overline{b}$ and $\ell$ are elements of $\mathfrak{W}$, $\overline{q}$ is an element of $\mathfrak{W}$.  Since $h$ is also an element of $\mathfrak{W}$, the pair $\left( \overline{q},h \right)$ is an element of $\mathfrak{W}$.  And commutativity of \eqref{eq_PushoutIso}, and the fact that $i \in \mathcal{K}$, yield
\[
\left(\overline{q},h \right) \sim_{\mathcal{K},\text{atomic}} \left( \overline{p}, \widetilde{f \restriction \mathfrak{N}}^{\ \mathcal{C}} \right).
\]

\end{proof}

\section{Purity and elementary submodels}

\subsection{Purity and traces}\label{sec_PurityAndTraces}

Let $\mathcal{C}$ be any category, and $\mu$ an infinite regular cardinal.  A morphism $f:  A \to B$ in $\mathcal{C}$ is (categorically) \textbf{$\boldsymbol{\mu}$-pure in $\boldsymbol{\mathcal{C}}$} in the sense of \cite[Definition 2.27]{MR1294136} if for every solid commutative diagram in $\mathcal{C}$ of the form
\[\begin{tikzcd}[ampersand replacement=\&,cramped]
	A \& B \\
	S \& T
	\arrow["f", from=1-1, to=1-2]
	\arrow[from=2-1, to=1-1]
	\arrow[from=2-1, to=2-2]
	\arrow[""{name=0, anchor=center, inner sep=0}, dotted, from=2-2, to=1-1]
	\arrow[from=2-2, to=1-2]
	\arrow["\circlearrowleft"{description}, draw=none, from=2-1, to=0]
\end{tikzcd}\]
with $S$ and $T$ both $\mu$-presentable in $\mathcal{C}$, there exists a $\mathcal{C}$-morphism from $T$ to $A$ making the lower left triangle commute.  This makes sense for any $\mathcal{C}$, and agrees with the more traditional definition of $\mu$-purity if $\mathcal{C}$ is the category of modules over a ring or acts over a monoid, where $\mu$-purity is typically defined in terms of solutions to $<\mu$ many equations with parameters from $A$.  Ordinary purity is $\omega$-purity, and pure embeddings are always monomorphisms in accessible categories \cite[Proposition 2.29]{MR1294136}.  If $\mathcal{C}$ is a Grothendieck category, $\mu$-purity of a morphism $f: A \to B$ is equivalent to the assertion that for every $\mu$-presentable object $X$, the canonical map
\[
\text{Hom}(X,B) \to \text{Hom}(X,\text{coker}(f))
\]
is surjective. 
	
Until further notice, ``pure" means in the sense above.  In Section \ref{sec_Qcoh_Kaplansky} we discuss the (generally weaker) notion of \emph{geometric purity} often used in the quasicoherent sheaf literature.
\begin{lemma}\label{lem_EasyPureTraces}
Suppose $\mathcal{C}$ is either $R$-Mod or $S$-Act, $\mu$ is an infinite regular cardinal, $\mathfrak{N}$ is appropriate for the relevant language, and $\mathfrak{N}$ is $<\mu$-closed.  Then $\mathfrak{N}$ is $\left( \text{Obj} \mathcal{C}, \mu \text{-pure}\right)$-appropriate.   
\end{lemma}
\begin{proof}
Fix a $\mu$-pure inclusion $f \in \mathfrak{N}$ and consider the square $\mathcal{S}(\text{Str} \mathcal{L}, f,\mathfrak{N})$ from page \pageref{eq_S_square}.  As discussed there, the $<\mu$-closure of $\mathfrak{N}$ ensures the vertical trace inclusions are $L_{\mu,\mu}$-elementary, so in particular they are $\mu$-pure.  To see that $f \restriction \mathfrak{N}: \mathfrak{N} \cap A \to \mathfrak{N} \cap B$ is $\mu$-pure, consider a $<\mu$-sized system $\mathcal{E}$ of equations with parameters from $\mathfrak{N} \cap A$ that is solvable in $\mathfrak{N} \cap B$.  The $<\mu$-closure of $\mathfrak{N}$ ensures $\mathcal{E} \in \mathfrak{N}$.  Then elementarity of $\mathfrak{N}$ and $\mu$-purity of $f: A \to B$ implies $\mathcal{E}$ is solvable in $\mathfrak{N} \cap A$.
\end{proof}

Lemma \ref{lem_CatPureTraces} is a variant of Lemma \ref{lem_EasyPureTraces} that will be used for the part of Theorem \ref{thm_MainApplications} dealing with (categorical) purity in Qcoh($X$).\footnote{A stronger but more complicated version of Lemma \ref{lem_CatPureTraces} can be proven, with weaker closure requirements on $\mathfrak{N}$, by invoking the notion of $\mu$-generated objects and the relevant theory from \cite{MR1294136}.}  A category $\mathcal{C}$ is \textbf{coherent in} $\text{Str} \mathcal{L}$ if for any $\text{Str} \mathcal{L}$-morphisms $f$ and $g$:  if $g$ and $gf$ are both in $\mathcal{C}$, then so is $f$.  Full subcategories of $\text{Str} \mathcal{L}$ are trivially coherent in $\text{Str} \mathcal{L}$.

\begin{lemma}\label{lem_CatPureTraces}

Suppose $\mathcal{C}$ is a coherent subcategory of $\text{Str} \mathcal{L}$, $\mu$ is an infinite regular cardinal, and $\mathfrak{N}$ is an $\mathcal{L}$-appropriate structure such that 
\begin{enumerate}[label=(\roman*)]

	\item $\text{Pres}_\mu(\mathcal{C}) \subset \mathfrak{N}$ (i.e., every $\mu$-presentable object in $\mathcal{C}$ is isomorphic to an element of $\mathfrak{N}$)
	\item $\mathfrak{N}$ is $\le |D|$-closed for every $D$ in the representative set $\text{Pres}_\mu(\mathcal{C})$.  (together with $D \in \mathfrak{N}$ this implies $D \subset \mathfrak{N}$)

	\item $\mathfrak{N} \prec_1 (V,\in, \mu, \mathcal{C})$.\footnote{In applications, take $\theta$ with $H_\theta$ closed under the relevant $\mathcal{C}$ operations.  This is formalizable in ZFC when $\mathcal{C}$ is a definable class (using L\'evy-Montague Reflection), or in G\"odel-Bernays class theory.}

\end{enumerate} 

Then
\begin{enumerate}
	\item\label{item_TraceIncPure} For any $C \in \mathfrak{N} \cap \text{Obj} \mathcal{C}$:  if the trace inclusion $\mathfrak{N} \cap C \subset C$ is in $\mathcal{C}$, then it is $\mu$-pure in $\mathcal{C}$.

	\item\label{item_RestrictionPure} For any $\mu$-pure $\mathcal{C}$-morphism $f \in \mathfrak{N}$:  if the square $\mathcal{S}(\text{Str} \mathcal{L}, f,\mathfrak{N})$ is in $\mathcal{C}$ (i.e., the vertical trace inclusions are in $\mathcal{C}$ and $f \restriction \mathfrak{N}$ is in $\mathcal{C}$) then all maps in it are $\mu$-pure.   
	
 \item\label{item_SpecialCaseFull} If $\mathcal{C}$ is a \textbf{full} subcategory of $\text{Str} \mathcal{L}$ and $\mathfrak{N} \cap C \in \text{Obj} \mathcal{C}$ for every $C \in \mathfrak{N} \cap \text{Obj} \mathcal{C}$, then $\mathfrak{N}$ is appropriate for the category 
 \[
 \left( \text{Obj} \mathcal{C}, \mu \text{-Pure} \right).
 \] 
\end{enumerate}

\end{lemma}
\begin{proof}

Part \ref{item_TraceIncPure}:  suppose $S$ and $T$ are $\mu$-presentable and the following is a commutative diagram in $\mathcal{C}$:

\[\begin{tikzcd}[ampersand replacement=\&]
	{\mathfrak{N} \cap C} \&\& C \\
	S \&\& T
	\arrow["\subset", from=1-1, to=1-3]
	\arrow["{i_C}"{description}, shift right=3, draw=none, from=1-1, to=1-3]
	\arrow["h", from=2-1, to=1-1]
	\arrow["g"', from=2-1, to=2-3]
	\arrow["k"', from=2-3, to=1-3]
\end{tikzcd}\]
We need to find a $\mathcal{C}$-morphism from $T$ to $\mathfrak{N} \cap C$ commuting with $g$ and $h$.

The assumptions ensure we can without loss of generality take $S$ and $T$ to each be \emph{elements and subsets} of $\mathfrak{N}$.  It follows by $\le |S|$-closure of $\mathfrak{N}$ that for any set function with domain $S$, if the image of that function is a subset of $\mathfrak{N}$, then the graph of the function (i.e., the function viewed as a set of ordered pairs, without mention of codomain) is an element of $\mathfrak{N}$.  In particular, $\text{graph}(g) \in \mathfrak{N}$, and since $T \in \mathfrak{N}$ too, $g \in \mathfrak{N}$.    And $\text{graph}(h) \in \mathfrak{N}$; since $C \in \mathfrak{N}$, it follows that $h^*:= i_C  h: S \to C$ is an element of $\mathfrak{N}$.

Now we cannot similarly conclude $k$ is a element of $\mathfrak{N}$, because $k$ doesn't necessarily map into $\mathfrak{N}$.  But $k$ witnesses the following assertion in the universe, where $H:=\text{Hom}_{\mathcal{C}}(T,C)$:
\[
(V,\in, \text{Obj} \mathcal{C}, \text{Mor} \mathcal{C}) \models \exists \ell \left( \ell \in H \text{ and }  \forall x \in S \  \ell g(x) = h^*(x) \right).
\]
All free parameters in that assertion are elements of $\mathfrak{N}$, so $\mathfrak{N}$ also satisfies the statement.  Say $\ell \in \mathfrak{N}$ is a witness.  Now $\ell \in \text{Hom}_{\mathcal{C}}(T,C)$, but since $\ell \in \mathfrak{N}$ and $T \subset \mathfrak{N}$, the image of $\ell$ is contained in $\mathfrak{N} \cap C$.  So there is a $\text{Str} \mathcal{L}$-morphism $\ell^*: T \to \mathfrak{N} \cap C$ with $i_C \ell^* = \ell$ (i.e., $\ell^*$ is just $\ell$ but with codomain $\mathfrak{N} \cap C$; this is clearly an $\mathcal{L}$-homomorphism).  Furthermore, $\ell^*  g = h$.  And since $\ell$ and $i_C$ are $\mathcal{C}$-morphisms and $i_C \ell^* = \ell$, coherency of $\mathcal{C}$ in $\text{Str} \mathcal{L}$ implies $\ell^*$ is a $\mathcal{C}$-morphism.    So $\ell^*$ is the desired $\mathcal{C}$-morphism from $T$ to $\mathfrak{N} \cap C$ commuting with $g$ and $h$.

Part \ref{item_RestrictionPure}:  suppose $S$ and $T$ are $\mu$-presentable, $f$ is $\mu$-pure with $f \in \mathfrak{N}$, and
\[\begin{tikzcd}[ampersand replacement=\&,cramped]
	C \&\& D \\
	{\mathfrak{N} \cap C} \&\& {\mathfrak{N} \cap D} \\
	S \&\& T
	\arrow["f", from=1-1, to=1-3]
	\arrow["{i_C}", from=2-1, to=1-1]
	\arrow["{f \restriction \mathfrak{N}}", from=2-1, to=2-3]
	\arrow["{i_D}"', from=2-3, to=1-3]
	\arrow["h", from=3-1, to=2-1]
	\arrow["g"', from=3-1, to=3-3]
	\arrow["k"', from=3-3, to=2-3]
\end{tikzcd}\]
is a commutative diagram in $\mathcal{C}$.  By similar arguments as above we conclude $g$, $h^*:= i_C  h$, and $k^*:= i_D  k$ are all elements of $\mathfrak{N}$ (note this time $k$ does map into $\mathfrak{N}$ so it is even easier than part \ref{item_TraceIncPure}).  Since $f$ is $\mu$-pure there is an $r: T \to C$ such that $rg = h^*$.  By elementarity of $\mathfrak{N}$ there is such an $r \in \mathfrak{N}$.  And since $\text{dom}(r) = T \subset \mathfrak{N}$, the image of $r$ is contained in $\mathfrak{N} \cap C$.  Let $r_0$ be the $\mathcal{L}$-homomorphism that agrees with $r$ on $T$ but has codomain $\mathfrak{N} \cap C$.  So $r = i_C r_0$ and by coherency of $\mathcal{C}$ in $\text{Str} \mathcal{L}$, $r_0$ is a $\mathcal{C}$-morphism.

Part \ref{item_SpecialCaseFull} follows immediately from Parts \ref{item_TraceIncPure} and \ref{item_RestrictionPure}, since full subcategories are trivially coherent.

\end{proof}

\section{Proof of Theorem \ref{thm_MainApplications}}

We focus first on $R$-Mod and $S$-Act (assuming $S$ is linearly orderd) and return to Qcoh($X$) later.

\subsection{The $R$-Mod and $S$-Act cases}\label{subsec_RMod_SAct_cases}

Assume $\kappa > |R|,|S|$, $\lambda^{<\kappa} = \lambda$, and that $\mathbf{K}$ is both $<\kappa$- and $\lambda$-Uniformly Stationary Kaplansky, as witnessed by stationary classes $\Gamma_{<\kappa} \subset [V]^{<\kappa}$ and $\Gamma_\lambda \subset [V]^\lambda$, respectively.  So
\begin{equation}\label{eq_WhatHoldsOnGammas}
(\forall \mathfrak{N} \in \Gamma_{<\kappa} \cup \Gamma_\lambda) \ (\forall K \in \mathfrak{N} \cap \mathbf{K}) \left( \mathfrak{N}\cap K \in \mathbf{K} \text{ and } \frac{K}{\mathfrak{N} \cap K} \in \mathbf{K} \right).
\end{equation}

Let $\mathcal{C}$ be either $R$-Mod or $S$-Act (with $S$ linearly ordered) and $\mathcal{L}$ be either the language of $R$-modules or the language of $S$-Acts.  Then $\mathcal{C}$ has pushouts, and every $\mathcal{L}$-appropriate structure is $\mathcal{C}$-appropriate since the trace inclusions are $\mathcal{L}$-elementary embeddings.  Let $\mathcal{K}$ be either $(\mathbf{K}, \text{pure})$ or $\left( \mathbf{K}, \mathbf{K}\text{-pure} \right)$; the $(\mathbf{K}, \text{Mono})$ case is an easier version that doesn't require linear orderability of $S$.  We verify that $\kappa$, $\lambda$, $\mathcal{K}$, and $\mathcal{C}$ satisfy the assumptions of Theorem \ref{thm_MainThm_Stability}.  Assumption \ref{item_ClosedUnderIso} follows from isomorphism-closure of $\mathbf{K}$, and assumption \ref{item_cardinal_arith} is a standing assumption in this section.  It remains to verify clauses \ref{item_K_lambda_stat_LS} and \ref{item_factoring}.  

We start with \ref{item_factoring}.  Suppose $\mathfrak{N} \in \Gamma_{<\kappa}$ and $f:K_1 \to K_2$ is in $\mathfrak{N} \cap \text{Mor} \mathcal{K}$.  Without loss of generality, $f$ is an inclusion.  Consider the diagram $\mathcal{D}(\mathcal{C},f,\mathfrak{N})$:
\[\begin{tikzcd}[ampersand replacement=\&]
	{K_1} \&\&\&\& {K_2} \\
	\\
	\&\& {P(\mathcal{C},f,\mathfrak{N})} \\
	\\
	{\mathfrak{N}\cap K_1} \&\&\&\& {\mathfrak{N} \cap K_2}
	\arrow["{f \in \mathfrak{N} \cap \text{Mor} \mathcal{K}}", from=1-1, to=1-5]
	\arrow["\subset"', shift right, draw=none, from=1-1, to=1-5]
	\arrow["{\widetilde{f \restriction \mathfrak{N}}^{\mathcal{C}}}"', from=1-1, to=3-3]
	\arrow["{f/^{\mathcal{C}} \mathfrak{N}}"{description}, from=3-3, to=1-5]
	\arrow["\ulcorner"{anchor=center, pos=0.125, rotate=-90}, draw=none, from=3-3, to=5-1]
	\arrow["\subset", from=5-1, to=1-1]
	\arrow["{f \restriction \mathfrak{N}}", from=5-1, to=5-5]
	\arrow["\subset"', from=5-5, to=1-5]
	\arrow[from=5-5, to=3-3]
\end{tikzcd}\]

By Fact \ref{fact_WhatD_looks_like} the pushout $P=P(\mathcal{C},f,\mathfrak{N})$ can be taken to be $K_1 + (\mathfrak{N} \cap K_2)$ in the $R$-Mod case, and $K_1 \cup (\mathfrak{N} \cap K_2)$ in the $S$-Act case, and all arrows can be taken to be inclusions.  Since $f$ is pure, Lemma \ref{lem_EasyPureTraces} ensures all maps in the outer rectangle are pure.  Since a pushout of a pure map is pure in any locally finitely presentable category \cite[Proposition 15]{MR2086721}, the maps pointing to the pushout are also pure.   And \eqref{eq_WhatHoldsOnGammas} ensures that the bottom corners, and the quotients of the vertical trace inclusions, all lie in $\mathbf{K}$.  Since pushouts preserve both cokernels and Rees quotients, the pure map from $\mathfrak{N} \cap K_2$ into $P$ has cokernel (Rees quotient) in $\mathbf{K}$.  Since $\mathbf{K}$ is closed under pure extensions and $\mathfrak{N} \cap K_2 \in \mathbf{K}$, it follows that 
\begin{equation}\label{eq_P_in_K}
P \in \mathbf{K}.
\end{equation}  
 
Next we claim that $f/^{\mathcal{C}} \mathfrak{N}$ is pure.
\begin{itemize}
	\item In the $R$-Mod case:  in order to show that the inclusion $f/^{\mathcal{C}} \mathfrak{N}: P  \to K_2$ is pure, it suffices to show that for any finitely-presented module $X$, the natural map from $\text{Hom}(X,K_2)$ to $\text{Hom}\left(X, K_2/P \right)$ is surjective.  The purity of $K_1$ in $K_2$ implies surjectivity of the natural map $\text{Hom}(X,K_2) \to \text{Hom}(X,K_2/K_1)$.  Since $K_1$ and $K_2$ are elements of $\mathfrak{N}$, so is $K_2/K_1$; and by Lemma \ref{lem_EasyPureTraces}, 
\begin{equation}\label{eq_N_cap_B_mod_A}
\mathfrak{N} \cap (K_2/K_1) \text{ is } \text{ pure in } K_2/K_1.
\end{equation}
So the natural homomorphism $\text{Hom}\left(X, K_2/K_1 \right) \to \text{Hom}\left(X, \frac{K_2/K_1}{\mathfrak{N} \cap (K_2/K_1)} \right) $ is surjective.  And by Fact \ref{fact_WhatD_looks_like},
\begin{equation}\label{eq_CalcQuotientOfPushout}
\frac{K_2}{P} \simeq \frac{K_2/K_1}{\mathfrak{N} \cap (K_2/K_1)}.
\end{equation}
So each map in the composition below is surjective, yielding purity of $P$ in $K_2$:
\begin{equation}\label{eq_TwoSurjections}
\mathrm{Hom}(X, K_2) \;\twoheadrightarrow\; \mathrm{Hom}(X, K_2/K_1) \;\twoheadrightarrow\; \mathrm{Hom}\!\Big(X,\, \underbrace{(K_2/K_1)\big/\!\big(\mathfrak{N}\cap(K_2/K_1)\big)}_{\simeq K_2/P}\Big)  
\end{equation}

	\item In the $S$-Act case:  the assumption that $S$ is linearly ordered ensures that a union of two pure subacts is pure.  Since $K_1$ and $\mathfrak{N} \cap K_2$ are each pure in $K_2$, this implies $P = K_1 \cup (\mathfrak{N} \cap K_2)$ is pure in $K_2$.  
\end{itemize}

Hence $\widetilde{f \restriction \mathfrak{N}}^{\mathcal{C}}$ and $f/^{\mathcal{C}} \mathfrak{N}$ are both pure, and by \eqref{eq_P_in_K} their domains and codomains are in $\mathbf{K}$.  So they are in the category $(\mathbf{K},\text{Pure})$.  This concludes the verification of clause \ref{item_factoring} of Theorem \ref{thm_MainThm_Stability} for $\mathcal{K} = (\mathbf{K}, \text{pure})$.  In the $(\mathbf{K}, \mathbf{K}\text{-Pure})$ case it remains to show that the cokernels (or Rees quotients) of $f \restriction \mathfrak{N}$ and $f/^{\mathcal{C}} \mathfrak{N}$ are in $\mathbf{K}$.  But in this case, since $f \in \mathbf{K}\text{-pure}$, $K_2/K_1 \in \mathbf{K}$.  So by \eqref{eq_WhatHoldsOnGammas} it follows that
\begin{equation}
\mathfrak{N} \cap \frac{K_2}{K_1} \in \mathbf{K} \text{ and } \frac{K_2/K_1}{\mathfrak{N} \cap (K_2/K_1)} \in \mathbf{K}.
\end{equation}
Those objects are, respectively, isomorphic to the cokernel (quotient) of $f \restriction \mathfrak{N}$ and $f/^{\mathcal{C}} \mathfrak{N}$, by Fact \ref{fact_WhatD_looks_like}.  And pushouts preserve cokernels and Rees quotients, so $\widetilde{f \restriction \mathfrak{N}}^{\mathcal{C}}$ has cokernel (Rees quotient) in $\mathbf{K}$.  So both $\widetilde{f \restriction \mathfrak{N}}^{\mathcal{C}}$ and $f/^{\mathcal{C}} \mathfrak{N}$ are morphisms in the category $(\mathbf{K}, \mathbf{K} \text{-pure})$.  This completes the verification of clause \ref{item_factoring} of Theorem \ref{thm_MainThm_Stability}.

To verify clause \ref{item_K_lambda_stat_LS} of Theorem \ref{thm_MainThm_Stability}, consider any $\mathfrak{N} \in \Gamma_\lambda$ and any $\mathcal{K}$-morphism $f \in \mathfrak{N}$.  The same argument as above (but ignoring the parts involving the pushout) shows that the outer square of $\mathcal{D}(\mathcal{C},f,\mathfrak{N})$---i.e., the square $\mathcal{S}(\text{Str} \mathcal{L}, f,\mathfrak{N})$--- lies in $\mathcal{K}$.  In other words, $\mathcal{K}$ has the LS property on $\Gamma_{\lambda}$.

\begin{remark}
A similar argument works with all instances of ``pure" replaced by ``$\mu$-pure", as long as the stationary subclasses of $[V]^{<\kappa}$ and $[V]^\lambda$ witnessing the $<\kappa$ and $\lambda$-Uniformly Stationary Kaplansky properties of $\mathbf{K}$ each have stationary intersection with the class of $<\mu$-closed structures.  
\end{remark}

\subsection{The Qcoh($X$) case}\label{sec_Qcoh_Kaplansky}

We first explain how Qcoh($X$) can be regarded as a subcategory of $\text{Str} \mathcal{L}$ for a certain signature $\mathcal{L}$.  This is due to an equivalence of categories due to Enochs-Estrada~\cite{MR2139915} (see also \cite{MR2964610}).  The definition of the category Qcoh($\mathcal{R}$) of ``quasicoherent $\mathcal{R}$-modules'' is summarized below:
\begin{itemize}
	\item $\mathcal{R}: \mathcal{Q} \to \text{Rings}$ is a flat diagram of rings, i.e., a fixed covariant functor on some small (i.e., set-sized) category $\mathcal{Q}$ satisfying a kind of flatness property.  Both $\mathcal{Q}$ and $\mathcal{R}$ depend on the scheme $X$, though the exact nature of this dependence and the precise meaning of flatness in this context is not needed here (see \cite{MR2139915} for details).

	\item $\mathcal{M}$ is an \emph{$\mathcal{R}$-module} if it assigns to each object $v$ of 	$\mathcal{Q}$ an $\mathcal{R}(v)$-module $\mathcal{M}(v)$, and to each arrow $a\colon v\to w$ in $\mathcal{Q}$ an $\mathcal{R}(v)$-linear map $\mathcal{M}(a)\colon\mathcal{M}(v)\to\mathcal{M}(w)$ (functorially, with $\mathcal{M}(w)$ regarded as an $\mathcal{R}(v)$-module along the ring homomorphism $\mathcal{R}(a)$).  An $\mathcal{R}$-module morphism $\mathcal{M}\to\mathcal{M}'$ is a family of $\mathcal{R}(v)$-linear maps $\mathcal{M}(v)\to\mathcal{M}'(v)$ commuting with the transition maps $\mathcal{M}(a)$.  \textbf{$\boldsymbol{\mathcal{R}}$-Mod} denotes the category of $\mathcal{R}$-modules with $\mathcal{R}$-module homomorphisms.

	\item Let $\mathcal{L}_{\mathcal{R}}$ be the signature with a sort for each object $v$ of $\mathcal{Q}$, on that sort the abelian-group symbols together with a unary operation symbol $\cdot_r$ for each $r\in\mathcal{R}(v)$, and for each arrow $a: u \to v$ of $\mathcal{Q}$ a unary function symbol $F_a$ with input sort $u$ and output sort $v$.  So $\mathcal{L}_{\mathcal{R}}$ is an $\text{Obj}\mathcal{Q}$-sorted, finitary signature.  Then $\mathcal{R}$-Mod can be viewed as a full subcategory of $\text{Str} \mathcal{L}_{\mathcal{R}}$: $\mathcal{R}$-modules correspond to those $\mathcal{L}_{\mathcal{R}}$-structures
	\[
	\mathfrak{M} = \left( \bigsqcup_{v \in \text{Obj} \mathcal{Q}} M_v, \ \left( +_v, \cdot_r \right)_{v \in \text{Obj}\mathcal{Q}, r \in \mathcal{R}(v) }, \  \left(F_a: M_{\text{dom}(a)} \to M_{\text{cod}(a)}\right)_{a \in \text{Mor} \mathcal{Q}} \right)
	\]
	such that each $(M_v,+_v, \cdot_{r})_{r \in \mathcal{R}(v)}$ is an $\mathcal{R}(v)$-module, each $F_a$ is an $\mathcal{R}(\text{dom}(a))$-linear map, $F_{\text{id}_v} = \text{id}_{M_v}$ for each object $v$, and $F_{a \circ b} = F_a \circ F_b$ whenever the composition $a \circ b$ is defined in $\mathcal{Q}$.\footnote{Every $\mathcal{L}_{\mathcal{R}}$-homomorphism between two such structures is automatically an $\mathcal{R}$-module homomorphism, because the $\mathcal{L}_{\mathcal{R}}$-homomorphism commutes with each $\cdot_r$ and each $F_a$.}

	\item An $\mathcal{R}$-module $\mathcal{M}$ is \textbf{quasicoherent} if for every arrow $a\colon v\to w$ in $\mathcal{Q}$ the induced map
	\[
		\mathcal{R}(w)\otimes_{\mathcal{R}(v)}\mathcal{M}(v)\longrightarrow\mathcal{M}(w),
		\qquad s\otimes x\mapsto s\cdot\mathcal{M}(a)(x),
	\]
	is an isomorphism of $\mathcal{R}(w)$-modules.  \end{itemize}

\textbf{Qcoh($\mathcal{R}$)} denotes the full subcategory of $\mathcal{R}$-Mod consisting of the quasicoherent $\mathcal{R}$-modules; it is therefore also a full subcategory of $\text{Str} \mathcal{L}_{\mathcal{R}}$.  If the flat diagram $\mathcal{R}$ of rings arises as in \cite{MR2139915} from a quasicompact, semiseparated scheme $X$, then Qcoh($X$) is equivalent to Qcoh($\mathcal{R})$, and is a Grothendieck category, with colimits computed coordinatewise; in particular, cokernels and pushouts are computed coordinatewise.

\begin{lemma}\label{lem_quasicoherence_downward}
If $\mathcal{R}: \mathcal{Q} \to \text{Rings}$ is a flat diagram of rings and $\mathfrak{N}$ is $\mathcal{L}_{\mathcal{R}}$-appropriate, then $\mathfrak{N}$ is appropriate for the category Qcoh($\mathcal{R}$). 
\end{lemma}

\begin{proof}
Since Qcoh($\mathcal{R}$) is a \emph{full} subcategory of $\text{Str}\,\mathcal{L}_{\mathcal{R}}$, it
suffices to show that $\mathfrak{N} \cap \mathcal{M}$ is quasicoherent whenever
$\mathcal{M} \in \mathfrak{N}$ is quasicoherent.  Since the trace inclusion 
$\mathcal{N}:=\mathfrak{N}\cap\mathcal{M} \subset \mathcal{M}$ is $\mathcal{L}_{\mathcal{R}}$-elementary, in particular $\mathcal{N}$ is an $\mathcal{R}$-module and the inclusion  reflects positive primitive formulas.
Fix $a: v \to w$ in $\mathcal{Q}$ and set $R:=\mathcal{R}(v)$, $S:=\mathcal{R}(w)$.

\emph{Surjectivity of $S \otimes_R \mathcal{N}(v) \to \mathcal{N}(w)$:}  Given
$y \in \mathcal{N}(w)$, quasicoherence of $\mathcal{M}$ yields $n<\omega$, $s_i \in S$ and
$x_i \in \mathcal{M}(v)$ with $y = \sum_{i<n} s_i \cdot F_a(x_i)$.  Thus $\mathcal{M}$ satisfies
the pp formula $\exists \overline{x}\,\big(y = \sum_{i<n} \cdot_{s_i} F_a(x_i)\big)$.  Since the
parameter $y$ is in $\mathcal{N}$ and $\mathcal{N}$ reflects pp formulas, the $x_i$ may be taken in
$\mathcal{N}(v)$.

\emph{Injectivity of $S \otimes_R \mathcal{N}(v) \to \mathcal{N}(w)$:}   Restricting to pp formulas of sort $v$ shows $\mathcal{N}(v)$ is a pure
submodule of $\mathcal{M}(v)$, so $S \otimes_R \mathcal{N}(v) \to S \otimes_R \mathcal{M}(v)$ is
injective.  Composing with the isomorphism $S \otimes_R \mathcal{M}(v) \simeq \mathcal{M}(w)$ given by quasicoherence of $\mathcal{M}$ shows that $s \otimes x \mapsto s\cdot\mathcal{M}(a)(x)$ is injective on
$S \otimes_R \mathcal{N}(v)$; since it factors through $\mathcal{N}(w)$, the canonical map
$S \otimes_R \mathcal{N}(v) \to \mathcal{N}(w)$ is injective.
\end{proof}

For the category Qcoh($\mathcal{R}$) there is a weaker (generally non-equivalent) version of purity than the categorical notion from Section \ref{sec_PurityAndTraces}.  A morphism $f: \mathcal{M} \to \mathcal{N}$ in Qcoh($\mathcal{R}$) is \textbf{geometrically (or stalkwise) $\boldsymbol{\mu}$-pure} if $f(v): \mathcal{M}(v) \to \mathcal{N}(v)$ is a $\mu$-pure embedding of $\mathcal{R}(v)$ modules, for every $v \in \text{Obj} \mathcal{Q}$.  When $\mathcal{R}$ arises as in Enochs-Estrada~\cite{MR2139915} from a quasicompact semiseparated scheme $X$, geometric purity corresponds under the Enochs--Estrada equivalence to geometric purity of Qcoh($X$): the monomorphisms in Qcoh($X$) that remain monic after tensoring with every quasicoherent sheaf.  This is the purity notion that the quasicoherent sheaf literature regards as the natural one (\cite{MR3682629}, \cite{MR3572762}).  We use $\text{Pure}_{geo}$ to denote this version of purity.

\begin{theorem}\label{thm_Precise_Qcoh}
Suppose $\mathcal{R}: \mathcal{Q} \to \text{Rings}$ is a flat diagram of rings, and $\mathbf{K}$ is an isomorphism-closed subclass of Qcoh($\mathcal{R}$) that is closed under geometrically pure extensions.  Suppose $\kappa$ is a regular uncountable cardinal larger than $\left\vert \mathcal{L}_{\mathcal{R}} \right\vert$.  If $\lambda$ is a cardinal, $\lambda^{<\kappa} = \lambda$, and $\mathbf{K}$ is both $<\kappa$-Uniformly Stationary Kaplansky and $\lambda$-Uniformly Stationary Kaplansky, then each of
\begin{equation}\label{eq_GeoPure}
\left( \mathbf{K}, \text{Pure}_{\text{geo}} \right), \left( \mathbf{K}, \mathbf{K}\text{-Pure}_{\text{geo}} \right), \left( \mathbf{K}, \text{Mono} \right)
\end{equation}
is $\lambda$-stable.  

Let $\text{FP}$ be a representative set of finitely presentable quasicoherent $\mathcal{R}$-modules and 
\[
\zeta:= \left| \text{FP} \cup \bigcup \text{FP} \right|^+.
\]
If the stationary subclasses of $[V]^{<\kappa}$ and $[V]^\lambda$ witnessing the $<\kappa$- and $\lambda$-Uniformly Stationary Kaplansky Property of $\mathbf{K}$ each contain stationarily many $<\zeta$-closed structures, then the conclusion holds also for categorical purity in Qcoh($X$) (and only assuming $\mathbf{K}$ is closed under categorically pure extensions).
\end{theorem}

The geometric part of Theorem \ref{thm_Precise_Qcoh} follows almost immediately from Lemma \ref{lem_quasicoherence_downward} together with the $R$-Mod version in Section \ref{subsec_RMod_SAct_cases}.  Let $\mathcal{C}=$Qcoh($\mathcal{R}$) and $\mathcal{K}$ be any of the categories listed in \eqref{eq_GeoPure}.  We verify clause \ref{item_factoring} of Theorem \ref{thm_MainThm_Stability}.  Let $\Gamma_{<\kappa}$ be the stationary subclass of $[V]^{<\kappa}$ witnessing that $\mathbf{K}$ is $<\kappa$-Uniformly Stationary Kaplansky.  Suppose $\mathfrak{N} \in \Gamma_{<\kappa}$ and $f \in \mathfrak{N}$ is a morphism in Qcoh($\mathcal{R}$) that is geometrically pure.  By Lemma \ref{lem_quasicoherence_downward}, the outer square of $\mathcal{D}(\mathcal{C},f,\mathfrak{N})$ is in Qcoh($\mathcal{R}$).  For each object $v \in \mathcal{Q} = \text{dom}(\mathcal{R})$ let $\mathcal{C}_v:= \mathcal{R}(v)$-Mod.  Since every such $v$ is an element of $\mathfrak{N}$, the $v$-th coordinate $f_v$ of $f$ is in $\mathfrak{N} \cap \mathcal{C}_v$, and is pure in $\mathcal{C}_v$.  By the argument in Section \ref{subsec_RMod_SAct_cases} for $R$-Mod (with $R:= \mathcal{R}(v)$), the arrows $\widetilde{f_v \restriction \mathfrak{N}}^{\mathcal{C}_v}$ and $f/^{\mathcal{C}_v} \mathfrak{N}$ in the square $\mathcal{D}(\mathcal{C}_v, f_v, \mathfrak{N})$ are pure in $\mathcal{C}_v$.  Since pushouts are computed coordinatewise in Qcoh($\mathcal{R}$), this implies $\widetilde{f \restriction \mathfrak{N}}^{\mathcal{C}}$ and $f/^{\mathcal{C}} \mathfrak{N}$ are both geometrically pure in Qcoh($\mathcal{R}$).    Since $\mathfrak{N} \in \Gamma_{<\kappa}$ the vertical trace inclusions in $\mathcal{D}(\mathcal{C},f,\mathfrak{N})$ have Qcoh($\mathcal{R}$)-cokernel in $\mathbf{K}$.  Since pushouts preserve cokernels and $\mathbf{K}$ is closed under (geometrically) pure extensions, we conclude that the pushout $P=P(\mathcal{C},f,\mathfrak{N})$ is in $\mathbf{K}$.  This completes the verification of Theorem \ref{thm_MainThm_Stability} clause \ref{item_factoring} for the category $(\mathbf{K}, \text{Pure}_{\text{geo}})$.  For the other classes we argue just as the $R$-mod case, using that cokernels in Qcoh($\mathcal{R}$) are computed coordinatewise.

Now turn to the case where purity is in the categorical sense of Section \ref{sec_PurityAndTraces}, and we make the additional assumption on $\mathbf{K}$ that the stationary subclasses of $[V]^{<\kappa}$ and $[V]^\lambda$ witnessing the Uniformly Stationary Kaplansky Properties each includes stationarily many $<\zeta$-closed structures.  Fix such an $\mathfrak{N}$ for the remainder of the argument.  Let $\mathcal{C}$ be Qcoh($\mathcal{R}$) and $\mathcal{K}$ be any one of 
\[
(\mathbf{K}, \text{Pure}), (\mathbf{K}, \mathbf{K}\text{-pure}), \text{ or } (\mathbf{K}, \text{Mono}),
\]
where ``pure" is in the categorical sense of Section \ref{sec_PurityAndTraces}, which is the convention for the remainder of this argument.  The rest of the argument looks identical to the argument in Section \ref{subsec_RMod_SAct_cases}, except we use Lemma \ref{lem_CatPureTraces} instead of Lemma \ref{lem_EasyPureTraces}.  Briefly:  suppose $f: K_1 \to K_2$ is a pure inclusion in Qcoh($\mathcal{R}$) with $f \in \mathfrak{N}$.  By $<\zeta$-closure of $\mathfrak{N}$ and Lemma \ref{lem_CatPureTraces}, all arrows in the outer square of $\mathcal{D}(\mathcal{C},f,\mathfrak{N})$ are pure, and it follows by closure of purity under pushouts in locally finitely presentable categories \cite[Proposition 15]{MR2086721} that the arrows pointing to $P=P(\mathcal{C},f,\mathfrak{N})$ are pure in Qcoh($\mathcal{R}$).  Since cokernels are preserved by pushouts, the arrow from $\mathfrak{N} \cap K_2$ into $P$ has cokernel in $\mathbf{K}$, and closure of $\mathbf{K}$ under pure extensions implies $P \in \mathbf{K}$.  By Lemma \ref{lem_CatPureTraces}, $\mathfrak{N} \cap (K_2/K_1)$ is pure in $K_2/K_1$; together with purity of $K_1$ in $K_2$ this implies that for any finitely presentable $X$, we have the following surjections:
\[
\text{Hom}(X,K_2) \twoheadrightarrow \text{Hom}(X,K_2/K_1) \twoheadrightarrow \text{Hom}\left( X, \frac{K_2/K_1}{\mathfrak{N} \cap (K_2/K_1)} \right)
\]  
Since cokernels are computed coordinatewise in Qcoh($\mathcal{R}$), Fact \ref{fact_WhatD_looks_like} implies
\[
\frac{K_2/K_1}{\mathfrak{N} \cap (K_2/K_1)} \simeq \frac{K_2}{K_1 + (\mathfrak{N} \cap K_2)} = \frac{K_2}{P}= \text{coker} \left( f/^{\mathcal{C}} \mathfrak{N} \right).
\]
Hence, $P$ is pure in $K_2$.  

The same kind of argument, but without dealing with pushouts, verifies clause \ref{item_K_lambda_stat_LS} of Theorem \ref{thm_MainThm_Stability}.  This takes care of the $(\mathbf{K}, \text{Pure})$ case; the others are almost verbatim the same argument as those in Section \ref{subsec_RMod_SAct_cases}.

\section{Proof of Corollary \ref{cor_SpecificExamples}}

Recall $\mathcal{FM}$ denotes the class of Flat Mittag-Leffler $R$-modules.  The following strengthens \v{S}aroch-Trlifaj~\cite[Section 3]{MR2988573} by bumping the conclusion from Kaplansky to Uniformly Stationary Kaplansky.
\begin{lemma}\label{lem_FM_UnifStatKap}
For any ring $R$ let $\kappa:= (|R|+\aleph_0)^+$,  $\lambda$ be regular with $\mu^{|R|+\aleph_0} < \lambda$ for all $\mu < \lambda$, and $\Gamma$ be the class of $\mathcal{L}_R$-appropriate models in $[V]^{<\lambda}$ that are $<\kappa$-closed (which is stationary, by Fact \ref{fact_LessKappaClosedStat}).  Then $\mathcal{FM}$ has the Kaplansky Property on $\Gamma$.  In particular, $\mathcal{FM}$ is a $<\lambda$-Uniformly Stationary Kaplansky class for $\lambda= \left( 2^{|R|+\aleph_0}\right)^+$.
\end{lemma}
\begin{proof}
Fix an $\mathfrak{N}\in \Gamma$ and consider any $A \in \mathfrak{N} \cap \mathcal{FM}$.  By Lemma \ref{lem_EasyPureTraces}, the trace inclusion $\mathfrak{N} \cap A \to A$ is $\kappa$-pure.  It follows by \cite[Lemma 4.1]{MR2900444} that $\mathfrak{N} \cap A$ is in $\mathcal{FM}$ (this just uses $\omega$-purity of the trace inclusion) and by the argument of the 3rd paragraph of the proof of \cite[Theorem 3.3]{MR2988573} that $\frac{A}{\mathfrak{N} \cap A}$ is in $\mathcal{FM}$ (this uses the $\kappa$-purity of the trace inclusion). 

\end{proof}

As in \cite[Theorem 3.3]{MR2988573} the argument given above actually works for the class of $\mathcal{Q}$-Mittag Leffler modules, where $\mathcal{Q}$ is any class of modules.  The class $\mathcal{FM}$ is the special case where $\mathcal{Q}$ is the class of flat modules.

Following the terminology of \cite{MR2964610}, an object $D \in \text{Qcoh}(\mathcal{R})$ is a \textbf{Drinfeld vector bundle} if $D(v)$ is a Flat Mittag-Leffler $\mathcal{R}(v)$-module for every $v \in \text{Obj}\left( \text{dom} \mathcal{R} \right)$.  If $\mathfrak{N}$ is $\mathcal{L}_{\mathcal{R}}$-appropriate, $D \in \mathfrak{N}$, and $\mathfrak{N}$ is $\le \mathcal{L}_{\mathcal{R}}$-closed, then $D(v) \in \mathfrak{N}$ for every $v \in \text{Obj} \mathcal{Q}$.  So by Lemma \ref{lem_FM_UnifStatKap}, both $D(v) \cap \mathfrak{N}$ and $\frac{D(v)}{\mathfrak{N} \cap D(v)}$ are both in $\mathcal{FM}_{\mathcal{R}(v)}$ for every $v$.  Hence $\mathfrak{N} \cap D$ and $\frac{D}{\mathfrak{N} \cap D}$ are both Drinfeld vector bundles.  This shows:

\begin{corollary}\label{cor_Drinfeld_UnifStatKap}
For any flat diagram $\mathcal{R}$ of rings let $\kappa:= (|\mathcal{L}_{\mathcal{R}}|+\aleph_0)^+$,  $\lambda$ be regular with $\mu^{|\mathcal{L}_{\mathcal{R}}|+\aleph_0} < \lambda$ for all $\mu < \lambda$, and $\Gamma$ be the class of $\mathcal{L}_R$-appropriate models in $[V]^{<\lambda}$ that are $<\kappa$-closed (which is stationary, by Fact \ref{fact_LessKappaClosedStat}).  Then the class of Drinfeld vector bundles has the Kaplansky Property on $\Gamma$.  In particular, the class of  Drinfeld vector bundles is a $<\lambda$-Uniformly Stationary Kaplansky class for $\lambda= \left( 2^{|\mathcal{L}_{\mathcal{R}}|+\aleph_0}\right)^+$.
\end{corollary}

Finally we prove Corollary \ref{cor_SpecificExamples}.  The class $\mathcal{FM}$ is closed under extensions in $R$-Mod \cite[Lemma 4.1]{MR2900444}.  And since cokernels are computed coordinatewise, the class of Drinfeld
vector bundles is closed under extensions in $\mathcal{R}$-Mod.  In particular they are closed under pure extensions, in both the categorical and the geometric sense.  It remains to verify the remaining hypotheses of Theorem
\ref{thm_MainApplications} (resp.\ Theorem \ref{thm_Precise_Qcoh}) for appropriate pairs of cardinals.

\textbf{The $\mathcal{FM}$ case:}  Let $\theta:= |R| + \aleph_0$ and $\kappa:= \left( 2^\theta
\right)^+$.  Then $\mathcal{FM}$ is $<\kappa$-Uniformly Stationary Kaplansky by Lemma \ref{lem_FM_UnifStatKap}.  Let $\lambda$ be any cardinal with $\lambda^{2^{\theta}} = \lambda$.  Then $\lambda^{<\kappa}  = \lambda$, and to apply Theorem \ref{thm_MainApplications} it remains to show that $\mathcal{FM}$ is $\lambda$-Uniformly Stationary Kaplansky, i.e., $<\lambda^+$-Uniformly Stationary Kaplansky.  By Lemma \ref{lem_FM_UnifStatKap} it suffices to know that $\mu^\theta < \lambda^+$ for all $\mu < \lambda^+$.  But this holds because $\mu^\theta \le \lambda^\theta = \lambda < \lambda^+$.  Theorem \ref{thm_MainApplications} now yields the $\mathcal{FM}$ part of Corollary
\ref{cor_SpecificExamples}.

\textbf{The Drinfeld vector bundle case:}  Let $\text{FP}$ and $\zeta = \left| \text{FP} \cup \bigcup
\text{FP} \right|^+$ be as in the statement of (the second part of) Theorem \ref{thm_Precise_Qcoh}.  Let $\theta:= \left| \mathcal{L}_{\mathcal{R}} \right| + \zeta + \aleph_0$, $\kappa:=
\left( 2^\theta \right)^+$, and let $\lambda$ be any cardinal with $\lambda^{2^\theta} = \lambda$.  The verification is
verbatim as above, using Corollary \ref{cor_Drinfeld_UnifStatKap} in place of Lemma
\ref{lem_FM_UnifStatKap}, and noting $\kappa > \left| \mathcal{L}_{\mathcal{R}} \right|$ as
required by Theorem \ref{thm_Precise_Qcoh}.  This gives the conclusion for geometric
purity.  For categorical purity we must further check that the witnessing stationary
classes contain stationarily many $<\zeta$-closed structures; but the classes produced by
Corollary \ref{cor_Drinfeld_UnifStatKap} consist of $<\theta^+$-closed structures, and
$\theta \ge \zeta$, so this is automatic.

\end{document}